\documentclass{article}
\usepackage[utf8]{inputenc}
\usepackage[english]{babel}
\usepackage{amsmath}
\usepackage{amsthm}
\usepackage{amsfonts}
\usepackage{amssymb}
\usepackage{cases}
\usepackage{hyperref}
\usepackage{esvect}
\usepackage{bbm}
\usepackage{geometry}[margin=3in]
\usepackage{accents}
\usepackage{graphicx}
\usepackage{caption}
\graphicspath{Translation flow.png}
\graphicspath{4-IET.png}
\graphicspath{Fejer kernel.png}

\newtheorem{thm}{Theorem}[section]
\newtheorem{lemma}[thm]{Lemma}
\newtheorem{cor}[thm]{Corollary}
\newtheorem{rmk}[thm]{Remark}
\newtheorem{prop}[thm]{Proposition}
\newtheorem{definition}[thm]{Definition}
\numberwithin{equation}{section}
\usepackage[english]{babel}
\usepackage[backend=biber, style=alphabetic, citestyle=alphabetic-verb]{biblatex}
\usepackage{csquotes}
\bibliography{paper.bib}
\newcommand\mymatrix{\left(\begin{smallmatrix}
1& 2& 3& 4\\
4& 3& 2& 1
\end{smallmatrix}\right)}

\title{QUANTITATIVE WEAK MIXING\\
FOR INTERVAL EXCHANGE TRANSFORMATIONS}
\author{Artur Avila, Giovanni Forni, Pedram Safaee}

\date{May 2021}

\begin{document}

\maketitle

\begin{abstract}

We establish a dichotomy for
the rate of the decay of the Ces\`aro averages of correlations of sufficiently regular
functions for typical interval exchange transformations (IET) which are
not rigid rotations (for which weak mixing had been previously established in the works of
Katok-Stepin, Veech, and Avila-Forni).  We show that the rate of decay is either logarithmic or
polynomial, according to whether the IET is of rotation class (i.e., it can be obtained as the
induced map of a rigid rotation) or not.  In the latter case, we also establish that the
spectral measures of Lipschitz functions have local dimension bounded away from zero (by a
constant depending only on the number of intervals).  In our approach, upper bounds are obtained
through estimates of twisted Birkhoff sums of Lipschitz functions, while the logarithmic lower
bounds are based on the slow deviation of ergodic averages that govern the relation between
rigid rotations and their induced maps.

\end{abstract}

\tableofcontents{}

\section{Introduction}

An \textit{interval exchange transformation} on $d$ intervals or a $d$-IET is a piecewise isometry that acts by rearranging the subintervals of a partition of an interval into $d$ subintervals according to a given permutation. The ergodic theory of IETs has been studied extensively since their introduction in the 1960's (Arnol'd
\cite{Arnol_d_1963}, Katok and Stepin \cite{MR0219697}, Oseledets
\cite{oseledec1968multiplicative}). 

A typical IET (with respect to the Lebesgue measure on the vector of the lengths of its subintervals) is minimal and uniquely ergodic (see \cite{MR357739} for minimality, and  \cite{10.2307/1971391}, \cite{10.2307/1971341} for unique ergodicity).  Katok \cite{MR594335} showed that IETs and 
special flows with bounded variation roof functions over IETs are never mixing.  In view of those results,
it became natural to investigate the frequency of the weak mixing property among IETs.  In what
follows, it will be convenient to keep in mind two (of many) equivalent characterizations of a weak mixing transformation.  The first is as a most direct weakening of the notion of mixing:
correlations of arbitrary $L^2$ functions must decay at least in the Ces\`aro sense.  The second is
in terms of the Lebesgue decomposition of spectral measures: they must be continuous, i.e., there can be no measurable eigenfunctions.  Naturally, for an IET to be weak mixing the associated permutation
can not be itself a rotation (i.e., preserve the natural cyclic order).

The prevalence of weak mixing was established for almost all three-IETs by Katok and Stepin (see \cite{MR0219697}). Veech \cite{10.2307/1971391} provided a sufficient condition for a special flow with constant roof function over an IET to be 
weakly mixing in terms of an accelerated version of the Rauzy--Veech renormalization cocycle (corresponding to returns to a good subset) and went on to prove that almost all IETs are 
weakly mixing for a restricted class of permutations. Nogueira and Rudolph \cite{MR1477038} proved that almost every IET is 
topologically weak mixing, i.e., almost every IET does not admit continuous eigenfunctions.
Later, Avila and Forni \cite{MR2299743} developed a novel probabilistic parameter exclusion technique to
prove weak mixing for almost all non-rotation IETs.
In addition, they proved weak mixing for almost all translation flows on higher genus surfaces
(which are special flows over certain IETs)
by a simpler linear elimination argument based on the hyperbolicity of the renormalization cocycle \cite{MR1888794}.  We note that the case of translation flows is easier to deal with
due to the presence of extra parameters (corresponding to the suspension data), but if one wishes
to consider more restricted families of translation flows (e.g., Veech surfaces, see
\cite{MR3522612}) new problems arise.

These results led to attempts at obtaining quantitative versions of the weak mixing property for
typical translation flows and IETs, for instance, through understanding the local behavior of
spectral measures.  In 2014, Bufetov and Solomyak \cite{MR3773061} proved H\"older bounds for the spectral measures of translation flows on surfaces of genus $2$. Later, generalizing the ideas of \cite{MR2299743} and using the analysis of a twisted cocycle, Forni obtained H\"older bounds for the spectral measures of almost all translation surfaces in all strata via proving a spectral gap property with respect to the Masur-Veech measures for the twisted cocycle, see \cite{2019arXiv190811040F}. This method was implemented along with a Diophantine parameter exclusion in the spirit of Salem, Erd\"os, and Kahane, by Bufetov and Solomyak \cite{2019arXiv190809347B}, who proved H\"older bounds for translation surfaces in all strata using a symbolic approach.

In this paper, we return the focus to the case of IETs.  Somewhat surprisingly, our results reveal an important quantitative difference between two types of weak mixing IETs, which does not have a counterpart for translation surfaces.  For simplicity, let us present our results first in terms of the quantification of the decay of Ces\`aro averages
\begin{equation}
    C_N(f,g)=\frac {1} {N} \sum_{n=0}^{N-1} \left |\int f \circ T^n(x) g(x) dx\right | \,,
\end{equation}
where $T$ is the IET and $f$ and $g$ are zero average Lipschitz observables (restricting considerations to a class of regular observables is necessary; for general reasons no decay rate can be established for $L^2$ functions).  Note that an estimate as $C_N=o(N^{-1})$ would already imply mixing, so the decay is certainly not super-polynomial.

An IET is said to be of rotation class if it can be obtained as an induced map of a rigid
rotation, in other terms, the Rauzy class of its associated permutation must contain a rotation.
For instance, all three-IETs are of rotation class.
We show that typical non-rotation class IETs do display polynomial decay $C_N=o(N^{-\alpha})$
for some $\alpha>0$ depending only on the number of exchanged intervals.  On the other hand,
for typical (non-rotation) rotation class IETs we provide logarithmic upper and lower bounds, that is,
while there is some $a>0$ such that $C_N=o(\log^{-a} N)$, an estimate such as $C_N=o(\log^{-b} N)$
fails in general (for some larger $b>0$).

We note that the rotation class condition for an IET is equivalent to the associated translation surfaces (obtained by the zippered rectangles construction) having genus one.  However, translation flows in genus one are never weak mixing (they are indeed quasiperiodic), so the discussion of quantitative weak mixing for typical translation flows is somewhat less interesting, since it collapses to the polynomial case.

Our estimates are based on the development of a generalized version of the nonlinear parameter exclusion technique of Avila and Forni \cite{MR2299743}, and a quantitative version of Veech criterion formulated by Bufetov and Solomyak \cite{2019arXiv190809347B}.  Our more technical results include H\"older bounds for the spectral measures, in terms of which the main distinction between rotation class and non-rotation class is the uniformity with respect to the spectral parameter.
As in the work of Avila and Leguil \cite {MR3933880}, we are able to apply large deviation methods of Avila and Delecroix \cite{MR3522612} to obtain control on the Hausdorff dimension of the set of exceptional parameters.

We now turn to the precise statements.

\begin{thm}
\label{mm}
Let $\pi$ be an irreducible permutation on $d$ symbols which is not of rotation class. Then there exist $\gamma>0, \beta>0$ (depending only on the Rauzy class of $\pi$) such that, for a set of parameters $\lambda$ in $\mathbb{P}_{+}^{d-1}$ whose complement has Hausdorff dimension strictly less than 
$d-1$, there exists a constant $C_\lambda>0$ such that for any $0<t<1$, and any Lipschitz test function $f$ on $[0,1)$, and any $\theta \in [t,1-t]$,
\begin{equation}
    \left|\sum_{n=0}^{N-1}e^{2\pi\imath  n \theta} f(T_{\lambda, \pi}^{n}(x))\right| \leq C_\lambda \|f\|_{L} \cdot t^{-\beta}N^{1-\gamma}\,, \; \;\: \quad \forall N\geq 1,
\end{equation}
for all $x\in I$. The above implies that
\begin{equation}
    \sigma_f\left([\theta -r, \theta +r]\right) \leq C'_\lambda\|f\|^{2}_{L}t^{-2\beta} \cdot r^{2\gamma}, \;\; \quad \forall r \leq 1/2,
\end{equation}
where $C'_\lambda>0$ is a constant and the Lipschitz norm $\|\cdot \|_{L}$, and the spectral measure $\sigma_f$ of the function $f$ with respect to the interval exchange transformation $T_{\lambda, \pi}$ of parameters $(\lambda, \pi)$ are defined in section \ref{sec:twist_int}.
\end{thm}

\begin{definition}
For a Borel measure $\sigma$ defined on $\mathbb{R}$ or $\mathbb{S}^{1}$, the
{\em lower local dimension} at a point $\theta \in \mathbb{R}$ $(\mathbb{S}^{1})$ denoted by $\underline{d}(\sigma, \theta)$ is given by

\begin{equation}
\underline{d}(\sigma, \theta):= 
\liminf_{r \to 0}{\frac{\log \sigma\left((\theta-r, \theta+r)\right)}{\log r}}.
\end{equation}
\end{definition}
\begin{definition} The Hausdorff dimension of a Borel measure $\sigma$ on $\mathbb{R}$ or $\mathbb{S}^{1}$ is given by
\begin{equation}
    \text{HD}(\sigma):=\sup \left\{s \geq 0: \underline{d}(\sigma, \theta) \geq s, \text{for} \; \sigma \;\text{a.e}\; \theta \right\}
\end{equation}

\end{definition}

\begin{thm} 
\label{specdim}
Under the assumptions of Theorem~\ref{mm}, we get that for all $\theta \notin \{0,1\}$,

\begin{equation}
    \underline{d}(\sigma_f, \theta) \geq 2\gamma,
\end{equation}
and 
\begin{equation}
    \text{HD}(\sigma_f) \geq 2 \gamma.
\end{equation}
\end{thm}

\begin{thm} 
\label{poldecave}
For every irreducible permutation $\pi$ on $d$ symbols which is not of rotation class, there exists a measurable subset $\Delta_{cor}$ of $\Delta=\mathbb{PR}_{+}^{d-1}$ with positive Hausdorff codimension and there exists $\alpha'>0$, depending only on the Rauzy class of $\pi$, such that for every $\lambda \in \Delta_{cor}$ there exists a constant $C_1(\lambda)>0$ such that for every zero-average Lipschitz function $f$, and for every $L^{2}$ function $g$ on $I$ the following holds
\begin{equation}
    \frac{1}{N}\sum_{n=0}^{N-1} \left|\left\langle f\circ T_{\lambda, \pi}^{n}, g \right\rangle\right|^{2} \leq C_1(\lambda)\|f\|_{L} \|g\|_{2}^2   \|f\|_{2}N^{-\alpha'}\,,  \; \;\: \quad \forall N \geq 1\,.
\end{equation}

\end{thm}

For rotation type IETs, in subsection~\ref{upper bound rotation class} we establish a logarithmic decay rate (upper bound), and in subsection~\ref{lower bound rotation class} we prove that no polynomial upper bound can be established in this case. In fact,
we construct Lipschitz test functions for which the Ces\`aro averages are bounded below by some logarithmic function. The precise results
can be stated as follows:

\begin{thm}
\label{rot_class_ub}
Let $\pi$ be a permutation on $d$ symbols of rotation class but which is not a rotation. Then for almost every $\lambda \in \mathbb{P}^{d-1}$, there exists a constant $C_2(\lambda)>0$ such that, for every zero-average Lipschitz function $f$, and every $L^{2}$ function $g$, we have the following  upper bound: for all integers $N\geq 2$, 

\begin{equation}
    \frac{1}{N}\sum_{n=0}^{N-1} \left|\big \langle f \circ T_{\lambda, \pi}^{n}, g \big \rangle\right|^{2}  \leq C_2(\lambda) \|f\|_{L} \|f\|_{2} \|g\|_{2}^{2}\frac{1}{(\log N)^{1/6}}
\end{equation}
where $T_{\lambda, \pi}$ is the IET corresponding to the parameters $(\lambda, \pi)$.
\end{thm}

\begin{rmk} The same ideas may be applied to self-similar IETs of type W (see definition~\ref{def:typeW} of section \ref{upper bound rotation class}) to obtain logarithmic upper bounds for the decay of the Ces\`aro averages of correlations.
\end{rmk}

\begin{thm}
\label{rot_class_lb}
For almost every rotation class interval exchange transformation $T$ there exists a constant $c_2(T)>0$ such that, for all integers $N\geq 3$ there exist Lipschitz functions $f$ and $g$ (depending on $N$) such that we have the following lower bound:
\begin{equation}
    Q_N(f, g) \geq c_2 \frac{N\|f\|_{L}^{2}\|g\|_{L}^{2}}{(\log N (\log \log N)^{1+\epsilon})^{9}},
\end{equation} 
where 
\begin{equation}
     Q_N(f, g):=\sum_{n=0}^{N-1} \left|\int_I f \circ T^{n}(y) g(y) dy - \int_I f \int_I g \right|^{2}.
\end{equation}
\end{thm}
We remark that all the upper bounds for the decay of correlations stated above may be extended to the space of H\"older observables at the expense of weakening the exponent of $N$. This may be done by a standard approximation of H\"older functions by Lipschitz functions based on convolution smoothing operators, see for instance Lemma 2.3 of \cite{https://doi.org/10.1002/cpa.3160280104}.

The above Theorems \ref{rot_class_ub} and \ref{rot_class_lb} apply in particular to the case of IETs on $3$-intervals. In this case the weak mixing property was first proved by 
A.~Katok and A.~M.~Stepin \cite{MR0219697} under an explicit full measure Diophantine condition. Later weaker and more precise conditions for weak mixing were established by S.~Ferenczi, C.~Holton and L.~Zamboni~\cite{FHZ}. It would be interesting to establish sufficient Diophantine conditions for logarithmic weak mixing.  

\smallskip
Another interesting class of cases where number theoretical considerations are crucial is given by self-similar IETs in any number of intervals. In this case a criterion for weak mixing was established by Ya.~Sinai and C.~Ulcigrai \cite{BSU}, \cite{SU} (based on a general criterion of the authors and A.~Bufetov \cite{BSU} for the singularity of spectral measures of IETs). Bufetov and Solomyak
\cite{bufetov2014modulus} established  polynomial (H\"older) bound for the spectrum of almost all suspensions over (non-Pisot, non-Salem) substitution systems
(under the condition that the leading eigenvalue of the substitution matrix has at least one Galois conjugate outside of the unit disk), but only logarithmic bounds in the self-similar case. Recently, polynomial bounds on spectral measures for special spectral parameters, in the Salem case, have been found by J. Marshall-Maldonado
\cite{2020arXiv2009.13607}. 
We are not aware of any result on the quantitative weak mixing of self-similar IETs.

\subsection{Structure of the paper}

In section~\ref{sec:preliminaries}, we begin with some preliminaries on cocycles and expanding maps. Then we introduce the notion of fast decay and derive from it Hausdorff dimension estimates for sets that are defined as the inferior limit of sets of a certain form. The main theorem in this section is Theorem~\ref{HDP}. In section~\ref{sec:exclusion}, we generalize a probabilistic technique of exclusion of parameters, developed originally by Avila and Forni \cite{MR2299743} to prove prevalence of weak mixing for IETs. Combining these ideas and the large deviation estimates of \cite{MR3522612}, we obtain a quantitative parameter exclusion, which is the content of 
Theorem~\ref{parexquan}. In section~\ref{sec:IETs}, we give a brief overview of IETs and the corresponding renormalization operator.  We restrict our attention to non-rotation type IETs (genus bigger than or equal to two) for which the renormalization operator has at least two positive Lyapunov exponents (with respect to the Masur-Veech measures) and apply the elimination argument (Theorem~\ref{parexquan}) to prove that the conditions of a quantitative version of the Veech's criterion for weak mixing are satisfied. In sections \ref{sec:twist_int} and \ref{sec:eff_Veech}, we follow the S-adic approach to IETs \cite{Bufetov2013}, \cite{bufetov2014IETs}. We derive uniform bounds on twisted Birkhoff sums of Lipschitz continuous functions and obtain the above-mentioned quantitative version of Veech's criterion for weak mixing (Theorem~\ref{QVC}). These sections essentially follow~\cite{MR3773061} and~\cite{2019arXiv190809347B} with some minor modifications to improve the upper bound and to remove the extra assumption of Oseledets regularity in their statement of the quantitative version of Veech criterion. In section \ref{sec:main_results}, we conclude the paper by deriving the main results stated in the introduction.
\section*{Acknowledgements}
During work on this project, Artur Avila was supported by the SNSF.

\section{Preliminaries}
\label{sec:preliminaries}

\subsection{Strongly expanding maps} Let $(\Delta, \mu)$ be a probability space and $T:\Delta \to \Delta$ a measurable transformation preserving the measure class of $\mu$. $T$ is said to be \textit{weakly expanding} if there exists a countable measurable partition $\{\Delta^{(\textit{l})}: \textit{l} \in \mathbb{Z} \}$ of $\Delta$ into sets of positive measure (with respect to $\mu$), such that for all $\textit{l} \in \mathbb{Z}$, $T$ maps $\Delta^{(\textit{l})}$ onto $\Delta$, $T^{(\textit{l})}:=T|_{\Delta^{(\textit{l})}}: \Delta^{(\textit{l})} \to \Delta$ is invertible, and $T^{(\textit{l})}_{*}\mu$ is equivalent to $\mu$.\\
Let $\Omega$ be the set of all words of finite length with integer entries and denote by $|\textbf{l}|$ the length of $\textbf{l} \in \Omega$. For any $\textbf{l}=(\textit{l}_{1}, \dots,\textit{l}_{n}) \in \Omega$, we define $\Delta^{\textbf{l}}:=\{ x \in \Delta: T^{k-1}(x) \in \Delta^{(\textit{l}_{k})}, 1 \leq k \leq n \}$ and $T^{\textbf{l}}:=T^{n}|_{\Delta^{\textbf{l}}}$. Then $\mu(\Delta^{\textbf{l}})>0$, and we let $\mathcal{M}:=\{ \mu^{\textbf{l}}: \textbf{l} \in \Omega \}$, where 

\begin{equation}
    \mu^{\textbf{l}}  :=\frac{1}{\mu(\Delta^{\textbf{l}})} T^{\textbf{l}}_{*}\mu.
\end{equation}
We say that $T$ is \textit{strongly expanding} if there exists $K>0$ such that $\forall \:\textbf{l} \in \Omega$
\begin{equation}
    \frac{1}{K} \leq \frac{d\mu^{\textbf{l}}}{d\mu} \leq K.  
\end{equation}
\begin{rmk}
If $T$ is strongly expanding, then for $Y \subset \Delta $, a set of positive $\mu-$measure, and $\nu \in \mathcal{M}$ the following bounded distortion estimates hold
\begin{equation}
    K^{-2}\mu(Y) \leq \frac{T^{\textbf{l}}_{*}\nu(Y)}{\mu(\Delta^{\textbf{l}})} \leq K^{2}\mu(Y).
\end{equation}
\end{rmk}
\subsection{Projective transformations}

Let $\mathbb{P}_{+}^{d-1} \subset \mathbb{P}^{d-1}$ be 
the projectivization of $\mathbb{R}_{+}^{d}$ and call it 
the $\textit{standard simplex}$.
A $\textit{projective}$ $\textit{contraction}$ is by definition the projectivizaiton of a matrix $B \in GL(d, \mathbb{R})$ with non-negative entries. 
The image of $\mathbb{P}_{+}^{d-1}$ under a 
projective contraction is called a $\textit{simplex.}$

\begin{lemma}
 \label{Lemma 1.2}
 (Lemma 2.1, \cite{MR2299743}) Let $\Delta$ be a simplex compactly contained in $\mathbb{P}_{+}^{d-1}$ and $\{\Delta^{(\textit{l})}: \textit{l} \in \mathbb{Z}\}$ a partition of $\Delta $ into sets of positive Lebesgue measure. Let $T:\Delta \to \Delta$ be a measurable transformation such that, for all $\textit{l}\in \mathbb{Z}$, $T$ maps $\Delta^{(\textit{l})}$ onto $\Delta$, $T^{(\textit{l})}:=T|_{\Delta^{(\textit{l})}}$ is invertible and its inverse is the restriction of a projective contraction. Then $T$ preserves a probability measure $\mu$ which is absolutely continuous with respect to the Lebesgue measure and has a density which is continuous and positive in $\Bar{\Delta}$. Moreover, $T$ is strongly expanding with respect to $\mu$.
\end{lemma}

\subsection{Cocycles}
A \textit{cocycle} is a pair $(T,A)$ where $T:(\Delta, \mu) \to (\Delta, \mu)$ and $A:\Delta \to GL(d,\mathbb{R})$ are measurable maps, viewed as a linear skew product $(x,w) \to (T(x),A(x).w)$ on $\Delta \times \mathbb{R}^{d}$. Notice that, $(T,A)^{n}=(T^{n}, A_n)$ where 
\begin{equation}
    A_n(x)=A(T^{n-1}(x))\cdots A(x)\,, \; \; \;\quad \text{for all } \,n\geq 0.
\end{equation}
$(T,A)$ is called an \textit{integral cocycle} if $A(x) \in SL(d,\mathbb{Z})$ for $\mu-$almost every $x\in \Delta$.

$(T,A)$ is said to be $\textit{locally}$ 
$\textit{constant}$ if there exists a partition of 
$\Delta$ into sets of positive $\mu$-measure $\{ 
\Delta^{(l)}: l \in \mathbb{Z}\}$ such that $A$ attains
the constant value $A^{(l)}$ over $\Delta^{(l)}$. If there is no 
nontrivial subspace of $\mathbb{R}^{d}$ that is 
invariant under the action of the group generated by 
$(A^{(l)})_{l \in \mathbb{Z}}$ the cocycle is called $\emph{irreducible}$.

If $\mu$ is an ergodic invariant probability measure for $T$, and 
\begin{equation}
    \int_{\Delta} \ln{\|A(x)\|}d\mu(x) < \infty,
\end{equation}
then $(T,A)$ is called a \textit{measurable cocycle}.

Let
\begin{equation}
    E^{s}(x):=\left\{w \in \mathbb{R}^{d}, \underset{n\to \infty}{\lim}\|A_n(x).w\|= 0\right\},
\end{equation}
\begin{equation}
    E^{cs}(x):=\left\{w \in \mathbb{R}^{d}, \underset{n\to \infty}{\limsup}\|A_n(x).w\|^{1/n} \leq 1 \right\}.
\end{equation}

Then $E^{s}(x) \subset E^{cs}(x)$ are subspaces of
$\mathbb{R}^d$ (they are called the $\textit{stable
space}$ and the $\textit{central stable}$ $\textit{space}$,
respectively), and we have $A(x).E^{cs}(x)=E^{cs}(T(x))$,
$A(x). E^{s}(x)=E^{s}(T(x))$. If $(T,A)$ is a measurable
cocycle, then by Oseledet's theorem $\dim E^{s}(x), \dim E^{cs}(x)$ are constant almost
everywhere.

For a matrix $B \in GL(d, \mathbb{R})$ let 
\begin{equation}
    \|B\|_{0}:= \max \left\{\|B\|, \|B^{-1}\|\right\}.
\end{equation}

We say that $(T,A)$ is \textit{log-integrable} if it satisfies
\begin{equation}
    \int_{\Delta} \ln\|A(x)\|_{0}d\mu(x) < \infty.
\end{equation}
\subsection{Fast decay}

\label{sec fastdec}
Assume that $(T,A)$ is a locally constant cocycle and $T$
is weakly expanding with respect to the measurable
partition $\{\Delta^{(\textit{l})}: \textit{l} \in
\mathbb{Z}\}$. We say that $T$ is \textit{fast decaying} if there
exist $C_1>0, \alpha_1>0$ such that
\begin{equation}
    \label{1.10}
    \underset{\mu(\Delta^{(\textit{l})}) \leq
    \epsilon}{\sum} \mu(\Delta^{(\textit{l})}) \leq
    C_1\epsilon^{\alpha_1}\,, \; \; \;\quad \text{for all } \, 0<\epsilon<1\,,
\end{equation}
and we say that $A$ is fast decaying if there exist $C_2,
\alpha_2>0$ such that 
\begin{equation}
    \underset{\|A^{(\textit{l})}\|_{0} \geq n}{\sum} \mu(\Delta^{(\textit{l})}) \leq C_2n^{-\alpha_2} \,, \; \; \;\quad \text{for all } n\in \mathbb N\setminus\{0\} \,.
\end{equation}
Fast decay of $A$ implies
that $(T,A)$ its log-integrability. The \textit{cocycle}
$(T,A)$ is said to be \textit{fast decaying} whenever
both $T,A$ are fast decaying. Notice that a fast decaying cocycle is not only log-integrable but also satisfies 
\begin{equation}
    \int_{\Delta} \|A(x)\|^{\epsilon} d\mu(x)<\infty,
\end{equation}
for all $0<\epsilon< \alpha_2$ which provides the basis for the large deviation estimates of \cite{MR3522612} that we use in this paper.

\subsection{Hausdorff dimension}
In this subsection, we prove an estimate on the Hausdorff dimension of some
special sets for cocycles with fast decay which will later be used to show that the Hausdorff dimension of some exceptional sets is not full. 
\begin{thm}
 \label{Hausbound}
 (Theorem 27, \cite{MR3522612}). Let $\Delta \subset \mathbb{P}^{d-1}$ be a simplex and assume that $T:\Delta \to \Delta$ satisfies the
conditions of Lemma~\ref{Lemma 1.2} and is fast decaying. For $n\geq 1$ let $X_{n} \subset \Delta$ be a union of $\Delta^{\emph{\textbf{l}}}$ with $|\emph{\textbf{l}}|=n$, and define $X:=\underset{n\to \infty}{\liminf}{X_n}$. If
\begin{equation}
     \underset{n \to \infty}{\limsup}-\frac{1}{n}\ln \mu(X_{n})>0,
\end{equation}
then
\begin{equation}
    \emph{HD}(X)<d-1,
\end{equation}
where $\emph{HD}(X)$ stands for the Hausdorff dimension of $X$.
\end{thm}
We will generalize this result in a way that is more appropriate for the applications in this paper. We will need the following lemma from \cite{MR3522612}.

\begin{lemma}
\label{Lemma 1.4}
(Lemma 28, \cite{MR3522612}) Assume that $T: \Delta \to \Delta$ has bounded distortion and is fast decaying with fast decay constant $\alpha_1>0$ (see. \eqref{1.10}). Then for $0<\alpha'<\alpha_1$ there exists $C_3>0$, depending on $\alpha'$, such that for every $n\geq 1$ we have,
\begin{equation}
    \underset{|\emph{\textbf{l}}|=n}{\sum} \mu(\Delta^{\emph{\textbf{l}}})^{1-\alpha'} \leq C_3^{n}.
\end{equation}
\end{lemma}
We will now state and prove the main theorem of this subsection whose proof closely follows the ideas of the proof of \textit{Theorem 27} in \cite{MR3522612}.
\begin{thm}
\label{HDP}
Let $\Delta \subset \mathbb{P}^{d-1}$ be a simplex and assume that $T: \Delta \to \Delta$ satisfies the conditions of Lemma~\ref{Lemma 1.2} and is fast decaying. For $n\geq 1$, let $X_n=\underset{j \geq n}{\bigcup} Y_{n,j}$ where $Y_{n,j}$ is a union of $\Delta^{\emph{\textbf{l}}}$ with $|\emph{\textbf{l}}|=j$, and define $X:=\underset{n \to \infty}{\liminf}X_n.$ If there exist $ \delta>0$, and $C>0$ such that for infinitely many $n$
\begin{equation}
\mu(Y_{n,j}) \leq Ce^{-\delta j } \;\; \; \forall j \geq n,
\end{equation}
then
\begin{equation}
    \emph{HD}(X)< d-1.
\end{equation}
\end{thm}
\begin{proof}
Note that there exists $C'>0$ such that if $0<r\leq r'$, then any 
simplex with Lebesgue measure at most $r'$ can be covered by 
$C'r'/r^{d-1}$ balls of diameter less than $r$. Let 
$\alpha':=\alpha_1/2$ and $C_3$ be as in Lemma~\ref{Lemma 1.4}. 
Let $0<\delta'<\alpha_1/4$ be small enough 
such that $C_3e^{-\delta \alpha'/4\delta'}<1$ and define 
$\lambda_1:=e^{-\delta/2\delta'}<1$. Assume $n$ is so that $\mu(Y_{n,j}) \leq Ce^{-\delta j}$ for all $j \geq n$. We will find a cover of $X_n$ with balls $B_{i}$ of diameter at most $\lambda_1^{n}$ in such a way that
\begin{equation}
    \underset{i}{\sum} \text{diam}(B_{i})^{d-1-\delta'} \leq C_0+C_1,
\end{equation}
for some constants $C_0, C_1$ which will be determined later.

We now let $V_j$ be
the union of all $\Delta^{\textbf{l}} \subset Y_{n,j}$ with 
$|\textbf{l}|=j$ such that $\mu(\Delta^{\textbf{l}})< \lambda_1^{j}$, 
and $W_j$ the complement of $V_j$ in $Y_{n,j}$. We can cover each 
$\Delta^{\textbf{l}} \subset V_j$ with at most 
$C'\mu(\Delta^{\textbf{l}})^{2-d}$ many balls of diameter $\mu(\Delta^{\textbf{l}})$. We call this cover $\{B_{i,j}^{V}\}$ and note that, by Lemma~\ref{Lemma 1.4},
\begin{equation}
\begin{split}
    \underset{j \geq n}{\sum}\underset{i}{\sum} \text{diam}(B_{i,j}^{V})^{d-1-\delta'} &\leq \underset{j\geq n}{\sum}\underset{\;|\textbf{l}|=j,\; \mu(\Delta^{\textbf{l}})<\lambda_1^{j}}{\sum}C'\mu(\Delta^{\textbf{l}})^{1-\alpha'}(\lambda_1^{j})^{\alpha'-\delta'}\\ & \leq
    C'\underset{j \geq n}{\sum}\left(C_3\lambda_1^{\alpha'/2}\right)^{j} \leq C_0,
\end{split}
\end{equation}
where $C_0:=C' (1-C_3\lambda_1^{\alpha'/2})^{-1}$. It can be readily verified that $W_j$ can be covered with at most $C'\mu(W_j)\lambda_1^{j(1-d)}$ balls of diameter $\lambda_1^{j}$. By letting this cover be $\{B_{i,j}^{W}\}$, we obtain
\begin{equation}
\begin{split}
    \underset{j\geq n}{\sum}\underset{i}{\sum} \text{diam}(B_{i,j}^{W})^{d-1-\delta'} &\leq \underset{j \geq n}{\sum}C'\mu(Y_{n,j})\lambda_1^{j(1-d)}(\lambda_1^{j})^{d-1-\delta'}\\ & \leq 
    C'C\underset{j \geq n}{\sum}\left(e^{-\delta/2}\right)^{j} \leq C_1,
\end{split}
\end{equation}
where $C_1:=\frac{CC'}{1-e^{-\delta/2}}$. It yields $\text{HD}(X) \leq d-1-\delta'<d-1$, as desired.
\end{proof}

\section{Exclusion of parameters}
\label{sec:exclusion}
 Let $(T,A)$ be a locally constant cocycle. We say that a compact set $\Theta \subset \mathbb{P}^{d-1}$ is $\textit{adapted}$ to the cocycle
$(T,A)$ if $A^{(\textit{l})}( \Theta) \subset \Theta$, for all 
 $\textit{l} \in \mathbb{Z}$, and for almost every
$x\in \Delta$, for every $w\in \mathbb{R}^{d}\setminus \{0\}$ that projectivizes to an element of $\Theta$, we have
\begin{equation}
    \|A(x)\cdot w\| \geq \|w\|,
\end{equation}
\begin{equation}
     \|A_{n}(x) \cdot w\| \to \infty.
\end{equation}
 
Let $\mathcal{J}=\mathcal{J}(\Theta)$ be the space of all
lines in $\mathbb{R}^{d}$ parallel to one of the elements
of $\Theta$ and not passing through the origin. 
The main result of this section is the following

\begin{thm}
\label{parexquan}
Assume that $(T,A)$ is a locally constant, irreducible, integral cocycle that is fast decaying. Let $\Theta$ be adapted to $(T,A)$ and assume that $J \cap E^{cs}(x)=\emptyset$, for every line $J \in \mathcal{J}(\Theta)$ and for almost every $x \in \Delta$. Let $h\in \mathbb{Z}^{d}$ be a primitive integer vector which projectivizes to an element of $\Theta$. Then, there exist $\epsilon>0$ and a subset $\Delta' \subset \Delta$, 
whose complement has Hausdorff dimension strictly less than 
full, such that the following holds. For all $ x \in \Delta'$ and $
t\in (0,1)$ there exists $n(t,x)$ so that $\forall
n \geq n(t,x)$ we have
\begin{equation}
   \frac{\#\left\{1\leq i \leq n: \|A_i(x)\cdot th\|_{\mathbb{R}^{d}/
    \mathbb{Z}^{d}}>\epsilon\right\}}{n}>\epsilon\,.
\end{equation}
Moreover $n(t,x)$ may be chosen to be of the form 
$C\max\{-\ln(t),-\ln(1-t)\}+K(x)$ for some constant $C>0$ (not depending on $x$) and a function $K(x)>0$.
\end{thm}
\begin{rmk}
The irreducibility assumption is only for simplicity and may in fact be dropped. Furthermore, the condition that $h$ is an integer vector can be removed.
\end{rmk}
We will prove this theorem by generalizing an exclusion of parameter technique developed by Avila and Forni in \cite{MR2299743} and applying some large deviation estimates for fast decaying 
cocycles introduced in the work of Avila and Delecroix in \cite{MR3522612}.
Henceforth, we will always assume that $(T,A)$ is a locally constant irreducible integral cocycle that is fast decaying.

For $J \in \mathcal{J}$, let $\|J\|$ be the distance from the origin to 
$J$. For $\delta>0$ small, let $\phi_\delta(\textbf{l},J)$ be the 
number of connected components of $A^{\textbf{l}} \cdot (J\cap B_{\delta}(0))\cap
B_{\delta}(\mathbb{Z}^{d}\setminus\{0\}).$ Let 
$J_{\textbf{l},0}:=A^{\textbf{l}}\cdot J$ and $J_{\textbf{l},k}$, $1\leq k \leq 
\phi_{\delta}(\textbf{l}, J)$, be all the lines of the form 
$A^{\textbf{l}}\cdot J-c$ where $c\in \mathbb{Z}^{d}\setminus\{0\}, 
A^{\textbf{l}}\cdot (J\cap B_{\delta}(0))\cap B_{\delta}(c)\neq \emptyset$. We 
define $\phi_{\delta}(\textbf{l}):=\underset{J\in 
\mathcal{J}}{\sup}\phi_{\delta}(\textbf{l},J)$. 

By formula $(3.6)$ of \cite{MR2299743} we have
\begin{equation}
    \label{2.4}
    \phi_{\delta}(\textbf{l}) \leq \|A^{\textbf{l}}\|_{0}.
\end{equation}

By definition $\|J_{\textbf{l},k}\| < \delta, \; k \geq 1 $ and from $(3.9)$ of \cite{MR2299743} we get the following lower bound

\begin{equation}
    \|J_{\textbf{l},k}\| \geq (1-2\delta) \|A^{\textbf{l}}\|_{0}^{-1} \geq 2^{-1} \|A^{\textbf{l}}\|_{0}^{-1}.
\end{equation}

For $k=0$ we have the following trivial bounds
\begin{equation}
    \|A^{\textbf{l}}\|_{0}\|J\| \geq \|J_{\textbf{l},0}\| \geq \|A^{\textbf{l}}\|_{0}^{-1}\|J\|.
\end{equation}

Let
\begin{equation}
    \mathbb{P_\nu}(X|Y)=\frac{\nu(X\cap Y)}{\nu(Y)}, \; \nu \in \mathcal{M},
\end{equation}
\begin{equation}
    \mathbb{P}(X|Y):=\underset{\nu \in \mathcal{M}}{\sup}\mathbb{P_\nu}(X|Y).
\end{equation}

Given $\eta>0$ small enough and $N$ large enough, by log-integrability of the cocycle one can find a finite set $Z \subset \Omega^{N}$ such that,
\begin{equation}
    \label{2.9}
    \mu\left(\bigcup_{\textbf{l}\in Z}\Delta^{\textbf{l}}\right)>1-\eta,
\end{equation}
\begin{equation}
    \label{2.10}
    \sum_{\textbf{l}\in \Omega^{N} \setminus Z}\ln\|A^{\textbf{l}}\|_{0}\mu(\Delta^{\textbf{l}}) <\frac{1}{2}.
\end{equation}

By viewing the cocycle as a random product of matrices with randomness given by $T$, the above condition assures that matrices in $Z$ occur with considerable probability. 
Since $Z$ is a finite set, we may take 
$\delta>0$ small enough such that,
\begin{equation}
\label{2.11}
\phi_{\delta}(\textbf{l})=0 \,, \;\;\; \forall
\textbf{l} \in Z.    
\end{equation}

The following proposition gives some uniform expansion for when our matrices are in $Z$.

\begin{prop}

(Claim 3.6, \cite{MR2299743}) For $N>0$ large enough there exists 
$\rho_{0}(Z)>0$ such that, for every $0<\rho 
<\rho_{0}(Z)$, every $J\in \mathcal{J}$ and 
every $Y \subset \Delta$ with $\mu(Y)>0$, we have
\begin{equation}
    \label{2.12}
    \underset{\nu \in \mathcal{M}}{\sup} 
    \underset{\emph{\textbf{l}}^{1} \in 
    Z}{\sum}\|J_{\emph{\textbf{l}}^{1},0}\|^{-\rho} 
    \mathbb{P}_{\nu} 
    \left(\Delta^{\emph{\textbf{l}}^{1}}\bigg| 
    \underset{\emph{\textbf{l}} \in Z}{\bigcup} 
    \Delta^{\emph{\textbf{l}}} \cap T^{-N}(Y)\right) 
    \leq (1-\rho) \|J\|^{-\rho}.
\end{equation}
\end{prop}
At this stage, we fix $\eta>0$ small enough, 
$0<\rho<\rho_{0}(Z), N \in \mathbb{N}$ large 
enough, and $Z\subset \Omega^{N}$ such that 
\eqref{2.9}, \eqref{2.10}, \eqref{2.11}, and \eqref{2.12} hold. Note that, for $\delta>0$ small enough, we have
\begin{equation}
    \label{3.13}
    \underset{\textbf{l}\in \Omega ^{N} 
    \setminus Z}{\sum}\left(\rho \ln 
    \|A^{\textbf{l}}\|_{0}+ \ln 
    (1+\|A^{\textbf{l}}\|_{0}(2\delta)^{\rho})\right)
    \mu(\Delta^{\textbf{l}})- 
    \rho\mu\left(\underset{\textbf{l} \in 
    Z}{\bigcup}\Delta^{\textbf{l}}\right)=\alpha
    (\delta) < \alpha <0.
\end{equation}

For $S\subset \{0,1,\dots,m-1\}$ we define 
\begin{equation}
    \Gamma^{m}_{\delta, S}(J):= \left\{x\in \Delta\Big| 
\exists
w \in J: 0\leq i \leq m-1, \|A_{iN}(x)\cdot w\|_{\mathbb{R}^{d}/\mathbb{Z}
^{d}}
\geq \delta  \Longleftrightarrow i \in S \right\}.
\end{equation}
Let 
\begin{equation}
    P_{\delta}(m):=\Big\{ S\subset \{0,1,\dots,m-1\},  0\notin S, \#S \leq
    m\delta\Big\}, 
\end{equation}
and 
\begin{equation}
    \Gamma^{m}_{\delta, \delta}(J):=\left\{x\in \Delta\Big|
    \exists w \in J: \|w\|_{\mathbb{R}^{d}/\mathbb{Z}^{d}} < \delta,\; \#\{0\leq i \leq m-1, \|A_{iN}(x)\cdot w\|_{\mathbb{R}^{d}/\mathbb{Z}^{d}} \geq
    \delta\}\leq m\delta \right\}.
\end{equation}
Note that
\begin{equation}
    \Gamma^{m}_{\delta, \delta}(J)= 
    \underset{S\in P_{\delta}(m)}{\bigcup} 
    \Gamma^{m}_{\delta, S}(J).
\end{equation}

Let $\Omega^{N},\hat{\Omega}^{N} \subset \Omega$
be the set of words of 
length $N$ in $\Omega$ and the set of words of 
length a multiple of $N$, respectively. We 
define $\psi: \Omega^{N} \to \mathbb{Z}$ such 
that $\psi(\textbf{l})=0$ iff $\textbf{l} \in Z$
and $\psi(\textbf{l})\neq \psi(\textbf{l}')$ 
whenever $\textbf{l}, \textbf{l}' \in \Omega^{N}
\setminus Z$, and  $\textbf{l}\neq \textbf{l}'$. We let $\Psi:\hat{\Omega}^{N}\to \Omega$ be given by $\Psi(\textbf{l}^{1}\dots \textbf{l}^{m})=\psi(\textbf{l}^{1}) \dots \psi(\textbf{l}^{m}) \in \Omega$, where $\textbf{l}^{i} \in \Omega^{N}$. For $\textbf{d}\in \Omega$ we let $\hat{\Delta}^{\textbf{d}}:=\bigcup_{\textbf{l} \in \Psi^{-1}(\textbf{d})}\Delta^{\textbf{l}}$ and for $S \in P_{\delta}(m)$ we take $C(\textbf{d},S), C(\textbf{d},\delta)$ to be the smallest real numbers such that 
\begin{equation}
    \mathbb{P}\left(\Gamma^{m}_{\delta, S}(J)\Big|\hat{\Delta}^{d}\right) \leq C(\textbf{d},S) \|J\|^{-\rho},
\end{equation}
\begin{equation}
    \mathbb{P}\left(\Gamma^{m}_{\delta,\delta}(J)\Big|\hat{\Delta}^{d}\right) \leq 
    C(\textbf{d}, \delta) \|J\|^{-\rho}.
\end{equation}

Since $Z$ is a finite set, we choose $M>0$ such that $\forall \textbf{l} \in Z$
\begin{equation}
     \label{2.20}
     \|A^{\textbf{l}}\|_{0}^{\rho}\left(1+\|A^{\textbf{l}}\|_{0}(2\delta)^{\rho}\right) \leq  \|A^{\textbf{l}}\|_{0}^{\rho}\left(1+\|A^{\textbf{l}}\|_{0}\right) \leq M(1-\rho).
\end{equation}

\begin{prop} Let $J\in \mathcal{J}$ with $\|J\|<\delta$, $\emph{\textbf{d}}=(d_1,\dots,d_{m-1})$, and $S$ a subset of $\{0,1,\dots,m-1\}$ not containing $0$. Then,
\begin{equation}
    \label{2.21}
    C(\emph{\textbf{d}},S) \leq (2\delta)^{-\rho\#S}M^{\#S} \underset{d_i=0}{\prod}(1-\rho) \: \underset{d_{i}\neq 0,\psi(\emph{\textbf{l}}^{i})=d_{i}}{\prod}\|A^{\emph{\textbf{l}}^{i}}\|_{0}^{\rho}\left(1+\|A^{\emph{\textbf{l}}^{i}}\|_{0}(2\delta)^{\rho}\right). 
\end{equation}
\end{prop}
\begin{proof}
Note that $0\notin S$ means that $\|J\|<\delta$. There are four possibilities:
\begin{enumerate}
    \item $1\notin S, d_{1}=0$ : We have by \eqref{2.10}, \eqref{2.11}, and \eqref{2.12} the bound
    \begin{equation}
        \mathbb{P}_{\nu}\left(\Gamma^{m}_{\delta, S}(J)\Big|\hat{\Delta}^{\textbf{d}}\right) \leq \underset{\textbf{l}^{1} \in Z}{\sum}\mathbb{P}\left(\Gamma^{m-1}_{\delta,S'}(J_{\textbf{l}^{1},0})\Big|\hat{\Delta}^{\textbf{d}'}\right)\mathbb{P}_{\nu}\left(\Delta^
        {\textbf{l}^{1}}\Big|\hat{\Delta}^{\textbf{d}}\right) \leq (1-\rho)C(\textbf{d}',S')\|J\|^{-\rho} \,,
    \end{equation}
    where $\textbf{d}'=(d_{2},\dots,d_{m-1})$, and $S'=\{i-1: i\in S\} \subset \{0,1,\dots,m-2\}$. 
    
    \item $1 \notin S, d_{1}\neq 0$: Let $\textbf{l}^{1} \in \Omega^{N}$ be so that $\psi(\textbf{l}^{1})=d_{1},$ then by \eqref{2.4} we have
    \begin{equation}
        \begin{split}
        \mathbb{P}_{\nu}\left(\Gamma^{m}_{\delta, 
        S}(J)\Big|\hat{\Delta}^{\textbf{d}}\right) &\leq 
        \sum_{k=0}^{\phi_{\delta}(\textbf{l}^{
        1})}\mathbb{P}\left(\Gamma^{m-1}_{\delta, 
        S'}(J_{\textbf{l}^{1},k})\Big|\hat{\Delta}
        ^{\textbf{d}'}\right)  \\ & \leq 
        C(\textbf{d}',S')\left(\|J_{\textbf{l}^{1},0}\|^{-\rho}+\phi_{\delta}(\textbf{l}^{1}
        ). \underset{k\geq  1}{\sup}\|J_{\textbf{l}^{1},k}\|^{-\rho}\right)\\
        & \leq 
        C(\textbf{d}',S')\|J\|^{-\rho}\quad\left( \|A^{\textbf{l}^{1}}\|_{0}^{\rho}+\frac{2^{\rho}
        \|A^{\textbf{l}^{1}}\|_{0}^{1+\rho}}{\|J\|^{-\rho}}\right) \\
        &\leq 
        C(\textbf{d}',S')\|J\|^{-\rho}\|A^{\textbf{l}^{1}}\|_{0}^{\rho}\left(1+\|A^{\textbf{l}^{1}}\|_{0}(2\delta)^{\rho}\right) \,,
        \end{split}
    \end{equation}
    where $\textbf{d}'=(d_2,\dots,d_{m-1})$, and $S'=\{i-1: i\in S\} \subset \{0,1,\dots,m-2\}$.

    \item $1\in S, \{1,\dots,m-1\} \not\subset S:$ Let $2<r\leq m-1$ be the smallest element of $\{2,\dots,m-1\}$ which is not in $S$. Let $\textbf{d}' := (d_{r+1}, \dots, d_{m-1})$, $S':=\{i-r: r \leq i \in S\}$ and $\textbf{l}':=\textbf{l}^{1}\dots\textbf{l}^{r}$ be the word realizing the supremum in the second line below. Then by \eqref{2.4}, and \eqref{2.20} we have
    \begin{equation*}
        \begin{split}
        & 
        \mathbb{P}_{\nu}\left(\Gamma^{m}_{\delta, 
        S}\Big|\hat{\Delta}^{\textbf{d}}\right)  \leq 
        \underset{\textbf{l}:\Psi(\textbf{l})=(d_1,\dots,d_r)}{\sum}\mathbb{P}_{\nu}\left(
        \Delta^{\textbf{l}}\Big|\hat{\Delta}^{\textbf{d}}\right)\sum_{k=0}^
        {\phi_{\delta}(\textbf{l})}\mathbb{P}\left(\Gamma^{m-r}_{\delta, 
        S'}(J_{\textbf{l},k})\Big|\hat{\Delta}
        ^{\textbf{d}'}\right)   \leq \\
        &\underset{\textbf{l}, \Psi(\textbf{l})=(d_1,\dots,d_r)}{\sup} \sum_{k=0}^
        {\phi_{\delta}(\textbf{l})}\mathbb{P}\left(\Gamma^{m-r}_{\delta, 
        S'}(J_{\textbf{l},k})\Big|\hat{\Delta}
        ^{\textbf{d}'}\right)
        \leq 
        C(\textbf{d}',S')\left(\|J_{\textbf{l}',0}\|^{-\rho}+\phi_{\delta}(\textbf{l}'
        )\cdot \underset{k\geq  1}{\sup}\|J_{\textbf{l}',k}\|^{-\rho}\right)\\
        & \leq 
        C(\textbf{d}',S')\|J\|^{-\rho}\quad\left( \|A^{\textbf{l}'}\|_{0}^{\rho}+\frac{2^{\rho}
        \|A^{\textbf{l}'}\|_{0}^{1+\rho}}{\|J\|^{-\rho}}\right) 
        \leq 
        C(\textbf{d}',S')\|J\|^{-\rho}\|A^{\textbf{l}'}\|_{0}^{\rho}\left(1+\|A^{\textbf{l}'}\|_{0}(2\delta)^{\rho}\right) \\
        &\leq 
        C(\textbf{d}',S')\|J\|^{-\rho} (2\delta)^{-(r-1)\rho} \underset{1\leq i\leq 
        r}{\prod}\|A^{\textbf{l}^{i}}\|_{0}^{\rho}\left(1+\|A^{\textbf{l}^{i}}\|_{0}(2\delta)
        ^{\rho}\right) \\
        &\leq  
        C(\textbf{d}',S')\|J\|^{-\rho} 
        (2\delta)^{-(r-1)\rho}\underset{\psi(\textbf{l}^{i})=d_i=0, i\leq r}{\prod} M(1-\rho)
        \underset{\psi(\textbf{l}^{i})=d_i\neq 
        0, i \leq r}{\prod}\|A^{\textbf{l}^{i}}\|_{0}^{\rho}\left(1+\|A^{\textbf{l}^{i}}\|_{0}(2\delta)^{\rho}\right)\\ &\leq
        C(\textbf{d}',S')\|J\|^{-\rho} 
        ( (2\delta)^{-\rho}M)^{r-1}\underset{\psi(\textbf{l}^{i})=d_i=0, i\leq r}{\prod} (1-\rho) 
        \underset{\psi(\textbf{l}^{i})=d_i\neq 
        0, i \leq r}{\prod}\|A^{\textbf{l}^{i}}\|_{0}^{\rho}
        \left(1+\|A^{\textbf{l}^{i}}\|_{0}(2\delta)^{\rho}\right).
        \end{split}
    \end{equation*}
  
    \item $\{1,\dots,m-1\} \subset S$:
    \begin{equation}
        \begin{split}
            C(\textbf{d},S) \leq 1   \leq (2\delta)^{-\rho\#S}M^{\#S} \underset{d_i=0}{\prod}(1-\rho) \: \underset{d_{i}\neq 0,\psi(\textbf{l}^{i})=d_{i}}{\prod}\|A^{\textbf{l}^{i}}\|_{0}^{\rho}\left(1+\|A^{\textbf{l}^{i}}\|_{0}(2\delta)^{\rho}\right), 
        \end{split}
    \end{equation}
    as $1 \leq M(1-\rho)$. The result follows by induction.
\end{enumerate}
\end{proof}
\noindent
Let

\begin{equation}
\gamma_{\delta}(x):=
\begin{cases}
 -\rho & \text{if}\; x \in \underset{\textbf{l} \in Z}{\bigcup} \Delta ^{\textbf{l}},\\
 \rho \ln \|A^{\textbf{l}}\|_{0}+ \ln (1+\|A^{\textbf{l}}\|_{0}(2\delta)^{\rho}) & \text{if}\; x \in \Delta^{\textbf{l}}, \textbf{l}\in \Omega^{N}\setminus Z.
\end{cases}
\end{equation}
\noindent
For $\delta>0$ small, by \eqref{3.13} we get
\begin{equation}
    \int_{\Delta}\gamma_{\delta}(x)d\mu(x)= \alpha(\delta)< \alpha<0 \,,
\end{equation}
and
\begin{equation}
\begin{split}
    \mathbb{P}\left(\Gamma^{m}_{\delta,\delta}(J)\Big|\hat{\Delta}^{\textbf{d}}\right) &\leq \underset{S\in P_{\delta}(m)}{\sum} \mathbb{P}\left(\Gamma^{m}_{\delta, S}(J)\Big|\hat{\Delta}^{\textbf{d}}\right) \\ &\leq \underset{S\in P_{\delta}(m)}{\sum} C(\textbf{d},S)\|J\|^{-\rho} \leq
    \# P_{\delta}(m) \underset{S \in P_{\delta}(m)}{\sup} C(\textbf{d},S)\|J\|^{-\rho}\,.
    \end{split}
\end{equation}
Therefore by Stirling's approximation (for the size of $P_{\delta}(m)$) and \eqref{2.21} there exists $c_0>0$ so that
\begin{equation}
    \label{2.29}
    C(\textbf{d},\delta) \leq  c_0(e\delta^{-1})^{\delta m}(2\delta )^{-\rho \delta m} M^{\delta m}\underset{d_i=0}{\prod}(1-\rho) \: \underset{d_{i}\neq 0,\psi(\textbf{l}^{i})=d_{i}}{\prod}\|A^{\textbf{l}^{i}}\|_{0}^{\rho}\left(1+\|A^{\textbf{l}^{i}}\|_{0}(2\delta)^{\rho}\right). 
\end{equation}

\noindent
Letting $C_m(x):=C(\textbf{d},\delta)$ for $x \in \hat{\Delta}^{\textbf{d}}, |\textbf{d}|=m$ and using \eqref{2.29} we obtain
\begin{equation}
    \frac{\ln C_m(x)}{m} \leq \frac{\ln(c_0)}{m}+\delta
    \ln(2Me)-(1+\rho) \delta 
    \ln(2\delta)+ \frac{1}{m}\sum_{i=0}
    ^{m-1} \gamma_{\delta}(T^{i}(x)).
\end{equation}
Thus, by taking $\delta>0$ small enough, we get that, for large $m$,
\begin{equation}
    \frac{\ln(c_0)}{m}+ \delta
    \ln(2Me)-(1+\rho) \delta 
    \ln(2\delta) \leq -\frac{\alpha}{4} \leq -\frac{\alpha(\delta)}{4}.
\end{equation}
Lemma 5.4 of~\cite{MR3933880} shows that for $S_{m}(x):=\sum_{i=0}^{m-1}\gamma_{\delta}(T^{i}(x))$, there exist $C,\beta>0$ such that
\begin{equation}
    \mu\left(\left\{x\in \Delta, \frac{S_m(x)}{m}>\frac{\alpha(\delta)}{2} \right\}\right) \leq Ce^{-\beta m}\,,
\end{equation}
hence, if $D_{m}:=\Big\{x \in \Delta, C_m(x)>e^{m\alpha(\delta)/4}\Big\}$, we have
\begin{equation}
    \label{2.33}
    \mu(D_m) \leq Ce^{-\beta m}.
\end{equation}
\noindent
By the bound in formula \eqref{2.33} and the fact that $C_m(x) \leq 1$, we get
\begin{equation}
    \label{2.34}
    \begin{split}
        \int_{\Delta}C_m(x)d\mu(x) &= \int_{D_m} C_m(x)d\mu(x)+\int_{\Delta \setminus D_m}C_m(x) d\mu(x)\\
        &\leq 
        \mu(D_m)+ \int_{\Delta \setminus D_m}e^{m\alpha(\delta)/4}d\mu(x) \leq 
        Ce^{-\beta m}+e^{m\alpha(\delta)/4}
         \leq 
        Ae^{-\kappa m}\,,
    \end{split}
\end{equation}
where $A:=C+1>0, \kappa:= \min\{\beta, -\alpha(\delta)/4 \}>0.$
We have therefore proved that
\begin{equation}
    \label{2.35}
    \begin{split}
    \mu\left(\Gamma^{m}_{\delta,\delta}(J)\right) &\leq \underset{\textbf{d}\in \hat{\Omega}^{N}, |\textbf{d}|=m}{\sum} \mu(\hat{\Delta}^{\textbf{d}})\mathbb{P}\left(\Gamma^{m}_{\delta, \delta}(J)\Big|\hat{\Delta}^{\textbf{d}}\right)\\
    &\leq 
    \int_{\Delta}C_m(x)\|J\|^{-\rho}d\mu(x) \leq Ae^{-\kappa m}\|J\|^{-\rho}.
    \end{split}
\end{equation}
\begin{rmk} For an integral log-integrable cocycle the same ideas can be applied to prove that $\mu(\Gamma^{m}_{\delta,\delta}(J)) \to 0$. However, to get exponential decay we have implicitly (see Lemma 5.4 in \cite{MR3933880}) made essential use of fast decay for the cocycle $(T,A)$.
\end{rmk}
Henceforth, we will fix $\delta>0$ small enough such that \eqref{2.35} holds.

\begin{definition}
Let $(T,A)$ be a {\em locally constant log-integrable 
cocycle} over a map $T: \Delta \to \Delta$ preserving a 
measure $\mu$. The expansion constant of $(T,A)$ 
is defined as the maximal real number $c\in \mathbb{R}$ such that 
$\forall v \in \mathbb{R}^{d}\setminus\{0\}$ and for 
$\mu-$almost every $x\in \Delta$ 
\begin{equation}
    \underset{n \to \infty}{\lim} \frac{1}{n}\ln\|A_{n}(x)\cdot v\| \geq c.
\end{equation}
The limit exists by Oseledets theorem applied to $\mu$.
\end{definition}
\begin{rmk}
When $A$ is irreducible, the expansion constant of $(T,A)$ is equal to the maximal Lyapunov exponent of the cocycle (see Lemma 26 in \cite{MR3522612}).
\end{rmk}
\begin{thm}
\label{LarDev}
(Theorem 25, \cite{MR3522612}) Let $T: \Delta \to \Delta$ be a 
countable shift endowed with a measure $\mu$ with 
bounded distortion that is fast decaying. Let $A:\Delta
\to SL(d,\mathbb{R}) $ be a fast decaying cocycle. Then
for every $c'$ smaller than the expansion constant of
$A$ there exist $C_3>0, \alpha_3>0$ such that, for every
unit vector $v \in \mathbb{R}^{d},$
\begin{equation}
    \mu\left(\left\{x \in \Delta: \|A_{n}(x)\cdot v\| \leq e^{c'n}\right\}\right) \leq C_3e^{-n\alpha_3} \,.
\end{equation}
\end{thm}

Let $(T,A)$ be as in Theorem~\ref{parexquan}. By combining Theorem~\ref{HDP}  and Theorem~\ref{LarDev}, we obtain the following corollaries.

\begin{cor}
\label{cor LD}

There exists a measurable $\Delta_0 \subset \Delta$ whose complement does not have full Hausdorff dimension, such that $\forall x \in \Delta_0,\; \exists k(x) \in \mathbb{N}$ such that $\forall k \geq k(x)$  

\begin{equation}
    e^{k\theta_1/2} \leq e^{k\theta_1 (1-1/4)} \leq \|A_{k}(x) \cdot h\| \leq e^{k\theta_1(1+1/4)} \leq e^{2k\theta_1}
\end{equation}
where $\theta_1$ is the maximal Lyapunov exponent of the cocycle $(T,A)$.
\end{cor}

Now if we let $h$ run through the elements of a basis for our vector space we get the following.

\begin{cor}
There exists a measurable set $\Delta'_0\subset \Delta$, whose complement does not have full Hausdorff dimension, such that $\forall x \in \Delta'_{0}, \exists k(x) \in \mathbb{N}$, such that $\forall k \geq k(x)$,
\begin{equation}
    e^{k\theta_1/2} \leq e^{k\theta_1(1-1/4)} \leq \|A_{k}(x)\| \leq e^{k\theta_1(1+1/4)} \leq e^{2k\theta_1}
\end{equation}
\end{cor}

\begin{rmk}
Although we do not make any use of the above corollary in this section, we will use it to show that the set of IETs to which the quantitative version of Veech criterion can be applied (which corresponds to the set $\Omega'_{\mathbf{q}}$ in the symbolic setting of Definition~\ref{defi}) has positive Hausdorff codimension. 
\end{rmk}

\begin{proof} [Proof of Theorem~\ref{parexquan}]

For $h \in \mathbb{Z}^{d}$ a primitive lattice point and $B \in SL(d, \mathbb{Z})$, let $J_h$ be the line spanned by $h$ and $J:= B\cdot J_h-c$, where $c\in \mathbb{Z}^{d} \setminus J_h$. Let $P$ be the triangle whose vertices are $\{0,c,Bh\}$. Then, since the minimal area of a lattice triangle is $1/2$, we obtain
\begin{equation}
    \frac{\|J\|\cdot\|Bh\|}{2}=\text{area}(P) \geq \frac{1}{2}.
\end{equation}
Thus,
\begin{equation}
    \label{2.38}
    \|J\| \geq \|Bh\|^{-1}.
\end{equation}
\noindent 
Let $\Delta_{0}$ and, for all $ x\in \Delta_{0} $,
$k(x) \in {\mathbb N}$ be as in Corollary~\ref{cor LD}. Let 
\begin{equation*}
\begin{split}
\mathcal{J}(n):=\Big\{J \in \mathcal{J}|& \exists x
\in \Delta_{0}, t 
    \in (0,1), c\in \mathbb{Z}^{d},\; \text{such 
    that}\;  \\ &  n \geq k(x),
     \|A_n(x) \cdot th-c\|< \delta, J=A_n(x) \cdot J_h-c\Big\}\,.
 \end{split}
 \end{equation*}
Every $J \in \mathcal{J}(n)$ can be uniquely determined by $c \in 
\mathbb Z^d$ and $$A_{n}(x)\cdot h \in \mathbb{Z}^{d} \cap B(0, e^{2n\theta_1})=:B_{\mathbb{Z}}(0,e^{2n\theta_1})\,,$$ where $B(0,r)$ is the ball of radius $r$ around $0$ in $\mathbb{R}^{d}$. Thus, there exists 
$C_0>0$ such that
\begin{equation}
    \label{2.43}
    \#\mathcal{J}(n) \leq \left(\#B_{\mathbb{Z}}(0,e^{2n\theta_1})\right)^{2} 
    \leq C_0e^{4nd\theta_1}.
\end{equation}
 By \eqref{2.38}, for every $J \in \mathcal{J}(n)$, we have that
\begin{equation}
    \label{2.44}
    \|J\|^{-\rho} \leq \|A_n(x)\cdot h\|^{\rho} \leq e^{2n\rho\theta_1 }.
\end{equation}
We choose $r \in \mathbb{N}$ large enough so that $4N d\theta_1+2N\rho \theta_1-r\kappa=-\alpha_2<0$ (where $N$ is as in \ref{3.13}). 
Let
\begin{equation}
    E_k:= \underset{m \geq rk}{\bigcup} T^{-kN}\Big(\underset{J \in \mathcal{J}(kN)}{\bigcup}\Gamma^{m}_{\delta, \delta}(J)\Big),
\end{equation}
and $F_n=\underset{k\geq n}{\bigcup} E_{k}$. Note that $F_n$ may be viewed as $\underset{j\geq n}{\bigcup} \Tilde{Y}_{n,j}$ where
\begin{equation}
    \Tilde{Y}_{n,j}:=
    \begin{cases}
     \underset{n\leq k \leq \frac{j}{r+1}}{\bigcup}T^{-kN}\Big(\underset{J \in \mathcal{J}(kN)}{\bigcup}\Gamma^{j-k}_{\delta, \delta}(J)\Big) &  \text{if} \; j\geq n+rn \\
     \emptyset & \text{otherwise.}
    \end{cases}
\end{equation}
Note that by  \eqref{2.35}, \eqref{2.43}, \eqref{2.44} we get that there exists $C>0$ such that
\begin{equation}
\begin{split}
  \mu(\Tilde{Y}_{n,j}) \leq &\sum_{n\leq k \leq \frac{j}{r+1}}\sum_{J \in \mathcal{J}(kN)} \mu\left(\Gamma^{j-k}_{\delta, \delta}(J)\right) \leq \sum_{1\leq k \leq \frac{j}{r+1}} C_0e^{4kNd\theta_1}Ae^{-\kappa (j-k)}e^{2kN\rho \theta_1} \\
  & \leq C_0Ae^{-\kappa j} \frac{j}{r+1}e^{(r\kappa-\alpha_2+\kappa)\frac{j}{r+1}} 
   \leq Ce^{-\frac{\alpha_2}{2(r+1)}j}\,.
\end{split}   
\end{equation}
For every $j,n$, $\Tilde{Y}_{n,j}$ is a union of $\Delta^{\textbf{l}}$ with $|\textbf{l}|=jN$. Therefore Theorem~\ref{HDP}, applied to the cocyle $(T^{N},A_{N})$, yields the following
\begin{equation}
   \text{HD}\left(F_\infty:=\underset{n}{\inf}F_n\right) <d-1.
\end{equation}

Letting $\Delta'$ be the intersection of $\Delta_0$ (defined in Corollary~\ref{cor LD}) and the complement of $F_\infty$, we get that for every $x \in \Delta'$ there exists $k_1(x)$ such that $\forall k \geq k_1(x)$ we have that
\begin{equation}
    \label{2.49}
     x\notin E_k \quad \text{and} \quad  e^{k\theta_1/2} \leq \|A_{k}(x)\cdot h\| \leq e^{2k\theta_1}
\end{equation}
and that the complement of $\Delta'$ has Hausdorff dimension less than full.

Thus for $t \in (0,1),\; \forall n \geq n_{0}(t,x):=\frac{2}{\theta_1} \max\{-\ln(t), -\ln(1-t)\}+k_1(x)$ we have
\begin{equation}
    \text{dist}\left(A_n(x) \cdot th, \left\{0,A_n(x)\cdot h\right\}\right) \geq \min\{t,1-t\} e^{n\theta_1/2} \geq 1.
\end{equation}
We let $n_1(t,x)$ be the smallest positive integer such that $n_1 N \ge n_0(t,x)$ and
\begin{equation} 
    \|A_{n_1 N}(x)\cdot th\|_{\mathbb{R}^{d}/\mathbb{Z}^{d}}< \delta.
\end{equation}
Therefore there exists $c\in \mathbb{Z}^{d}$ such that $J_{n_1(t,x)N}:= 
A_{n_1(t,x)N}.J_h-c \in \mathcal{J}(n_1(t,x)N)$. We fix $\epsilon< \min 
\{\frac{\delta}{4N}, \frac{1}{4rN}\}$ and 
$n(t,x):=\min\{2r(r+1)n_0(t,x),2rn_1(t,x)N\}$. From \eqref{2.49} we get that $T^{n_1(t,x)N}(x)\notin 
\Gamma^{m}_{\delta,\delta}(J_{n_1(t,x)N})$ for all $m \geq rn_1(t,x)N$, which easily 
implies that, for all $n\geq n(t,x)$, we have
\begin{equation}
   \frac{\#\left\{1\leq i \leq n: \|A_i(x)\cdot th\|_{\mathbb{R}^{d}/
    \mathbb{Z}^{d}}>\epsilon\right\}}{n}>\epsilon.
\end{equation}
Note that for $C:=\frac{4r(r+1)}{\theta_1}$, $K(x):=2r(r+1)k_1(x)$ we have
\begin{equation}
    n(t,x) \leq C\max\left\{-\log(t),-\log(1-t)\right\}+K(x)
\end{equation} 
and Theorem~\ref{parexquan} follows.
\end{proof}

\section{Interval exchange transformations}
\label{sec:IETs}
An \textit{interval exchange transformation} (IET) is a piecewise orientation-preserving 
isometry of an interval with finitely many singularities 
(discontinuities). In other words, let $I$ be an interval that is partitioned into $d>1$ subintervals
(the subintervals are assumed to be closed on the left and open on the 
right) labelled by the letters of an alphabet $\mathcal{A}$ consisting of $d>1$ elements. Let $\lambda \in \mathbb{R}_{+}^{d}$ and $\pi=(\pi_t,\pi_b)$
($t$ stands for top and $b$ for bottom) be 
a pair of bijective maps taking $\mathcal{A}$ to $\{1,2,\dots,d\}$. For each $\alpha \in 
\mathcal{A}$ let $I^{(0)}_{\alpha}$ be an 
interval of length $\lambda_\alpha$. Then 
the interval exchange transformation $T_{\lambda, \pi}:I \to I$ rearranges the intervals

\begin{equation}
I_{\pi_{t}^{-1}(1)}, I_{\pi_{t}^{-1}(2)}, \cdot \cdot \cdot, I_{\pi_{t}^{-1}(d)}, 
\end{equation}
via translations to the intervals

\begin{equation}
I_{\pi_{b}^{-1}(1)}, I_{\pi_{b}^{-1}(2)}, \cdot \cdot \cdot I_{\pi_{b}^{-1}(d)}.
\end{equation}
More precisely, for each $\alpha \in \mathcal{A}$ there exists $w_{\alpha} \in \mathbb{R}$ such that for every $x \in I_{\alpha}, T_{\lambda, \pi}(x)=x+w_{\alpha}$. Scaling the lengths of the intervals in an IET $T_{\lambda, \pi}$ gives rise to a conjugate dynamical system and therefore we may simply assume that $|\lambda|:=\sum_{\alpha\ \in \mathcal{A}} {\lambda_\alpha}=1.$ In order to study the statistical properties of IETs we will make use of an operation called Rauzy--Veech renormalization, defined below. 

\begin{figure}
\centering
\includegraphics[scale=0.4]{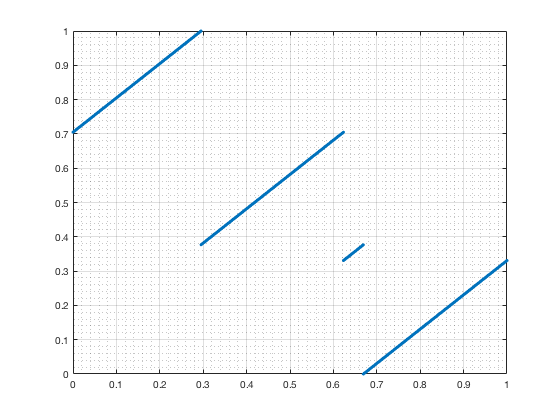}
\caption{The graph of a (pseudo) randomly generated 4-IET with  $\pi=\protect\mymatrix{}$.}
\end{figure}

\subsection{Renormalization algorithm}
By looking at the returns of an interval exchange to smaller and smaller subintervals, we may obtain information about longer and longer parts of orbits. This 
intuition is the key behind the idea of renormalization. It gives rise to a transformation on the parameter space, whose study sheds light on dynamical properties of large families (usually of full measure) of parameters.

\smallskip
Let $T_{\lambda^{(0)}, \pi^{(0)}}$ be an IET on the interval $I^{(0)}$ with parameters $(\lambda^{0}, \pi^{0})$
and $\alpha_{t}, \alpha_{b}$ be the letters with the property that $\pi_{t}(\alpha_{t})=\pi_{b}(\alpha
_{b})=d.$ We assume that $\lambda^{(0)}_{\alpha_{t}} \neq 
\lambda^{(0)}_{\alpha_{b}}$ and consider 
the first return map of $T_{\lambda^{(0)}, \pi^{(0)}}$ to the interval $I^{(1)}$ obtained by
removing the shorter of the  intervals $I_{\alpha_{t}}$ and $I_{\alpha_{b}}$ from the right end of $I^{(0)}$. The resulting transformation is an IET $T_{\lambda^{(1)}, \pi^{(1)}}$ defined as follows:
\begin{itemize}
\item If $\lambda^{(0)}_{\alpha_{t}}> \lambda^{(0)}_{\alpha_{b}}$, then $\lambda^{(1)}_{\alpha}=\lambda^{(0)}_{\alpha}$ for all $\alpha \neq \alpha_{t}, \indent \lambda^{(1)}_{\alpha_{t}}= \lambda^{(0)}_{\alpha_{t}}- \lambda^{(0)}_{\alpha_{b}}$, $\pi_{t}^{(1)}=\pi_{t}^{(0)}$, and
\begin{equation}
(\pi^{(1)}_{b})^{-1}(i)=
\begin{cases}
(\pi^{(0)}_{b})^{-1}(i) & \text{if}\; i \leq \pi^{(0)}_{b}(\alpha_t)\\
\alpha_{b} & \text{if}\;
i=\pi^{(0)}_{b}(\alpha_t)+1\\
(\pi^{(0)}_{b})^{-1}(i-1) & \text{otherwise}.
\end{cases}
\end{equation}
In this case $(\lambda^{(0)}, \pi^{(0)})$ is said to be of \textit{top type}
(since $\alpha_t$ ``wins'' against $\alpha_b$).
\item If $\lambda_{\alpha_{b}}^{(0)}>\lambda_{\alpha_{t}}^{(0)}$, then $\lambda_{\alpha}^{(1)}=\lambda_{\alpha}^{(0)}$ for all $\alpha \neq \alpha_{b}$, $\lambda_{\alpha_{b}}^{(1)}=\lambda_{\alpha_{b}}^{(0)}- \lambda_{\alpha_{t}}^{(0)}$, $\pi_{b}^{(1)}=\pi_{b}^{(0)}$, and 
\begin{equation}
(\pi^{(1)}_{t})^{-1}(i)=
\begin{cases}
(\pi^{(0)}_{t})^{-1}(i) & \text{if}\; i \leq \pi^{(0)}_{t}(\alpha_b)\\ 
\alpha_{t} & \text{if}\; i=\pi^{(0)}_{t}(\alpha_b)+1\\
(\pi^{(0)}_{t})^{-1}(i-1) & \text{otherwise}.
\end{cases}
\end{equation}
In this case $(\lambda^{(0)}, \pi^{(0)})$ is said to be of \textit{bottom type}
(since $\alpha_b$ ``wins'' against $\alpha_t$).
\end{itemize}
In the sequel, we will always assume that the combinatorial data $\pi$ is \textit{irreducible}, 
i.e., that there is no $1\leq k < d$ such that 
$\pi_{t}^{-1}(\{1,\dots,k\})=\pi_{b}^{-1}(\{1,\dots,k\})$. We denote the space of such combinatorial data by $\mathfrak{S}^{0}(\mathcal{A})$. The above operations define an equivalence relation on 
the set of combinatorial data. An 
equivalence class of this relation is called a \textit{Rauzy class} and will usually be 
denoted by $\mathfrak{R} \subset \mathfrak{S}^{0}(\mathcal{A})$. 
Therefore, we obtain a map $\mathcal{Q}_{R}: \mathbb{R}_{+}^{\mathcal{A}} \times \mathfrak{R} \to \mathbb{R}_{+}^{\mathcal{A}} \times \mathfrak{R}$ that sends $(\lambda^{(0)}, \pi^{(0)})$ to $(\lambda^{(1)}, \pi^{(1)})$, and is called the \textit{Rauzy 
induction map}. Rescaling $\lambda^{(1)}$ back to size $1$ 
(size being the $\ell^{1}$-norm) yields a map $\mathcal{R}_{R}: 
\mathbb{P}_{+}^{\mathcal{A}} \times \mathfrak{R} \to \mathbb{P}_{+}^{\mathcal{A}} \times \mathfrak{R}$, which we call the \textit{Rauzy renormalization map}. 

As defined above, an IET is not necessarily infinitely many times 
renormalizable, however, a condition that ensures infinite renormalizability of $T_{\lambda, \pi}$ is that the components of $\lambda$ satisfy no linear equation over $\mathbb{Q}$ other than the trivial one (which itself implies another condition called infinite distinct orbit condition, see \cite{MR357739}). It is easy to 
see that the set of those IETs not satisfying this condition has dimension $d-2$, and therefore it is no loss of generality for us to restrict our considerations to this set. Masur and Veech proved the following theorem independently in 1982. 
\begin{thm}(\cite{10.2307/1971341}, \cite{10.2307/1971391}) Let $\mathfrak{R} \subset \mathfrak{S}^{0}(\mathcal{A})$ be a Rauzy class. Then $\mathcal{R}_{R}\big|_{\mathbb{P}^{\mathcal{A}}_{+} \times \mathfrak{R}}$ admits an ergodic conservative infinite absolutely continuous invariant measure $\mu$, unique in its measure class up to a scalar multiple. Its density is a positive rational function.
\end{thm}
Later, Zorich defined an accelerated version of the Rauzy renormalization which is now called \textit{Zorich transformation}. For an element $(\lambda, \pi) \in \mathbb{P}_{+}^{\mathcal{A}} \times \mathfrak{R}$, $\mathcal{Q}_{Z}(\lambda, \pi)$ is $(\lambda^{(n)}, \pi^{(n)})$ where $n:=n(\lambda, \pi)$ is the smallest 
positive $n$ for which the type of $(\lambda^{(n)}, \pi^{(n)})=\mathcal{Q}_{R}^{n}(\lambda, \pi)$ is different from that of $(\lambda, \pi)$. The Zorich renormalization is defined by rescaling the intervals back to size $1$ and is denoted by $\mathcal{R}_{Z}$. Zorich showed the following.

\begin{thm} (\cite{AIF_1996__46_2_325_0}) Let $\mathfrak{R} \subset \mathfrak{S}^{0}(\mathcal{A})$ be a Rauzy class. Then $\mathcal{R}_{Z}|_{\mathbb{P}_{+}^{\mathcal{A}} \times \mathfrak{R}}$ admits a unique ergodic absolutely continuous probability measure $\mu_{Z}$. Its density is positive and analytic.
\end{thm}
\subsection{Rauzy Cocycle}
The Rauzy cocycle is a matrix cocycle that stores the visitation data of the renormalization procedure. In other words, for $(\lambda,\pi)$, $B^{R}(\lambda, \pi)$ is the matrix given by the following formula
\begin{equation}
    B^{R}(\lambda, \pi):=
    \begin{cases}
    I+ E_{\alpha_{b}\alpha_{t}} & \text{if} \; (\lambda, \pi)\; \text{is of top type}\\
    I+ E_{\alpha_{t}\alpha_{b}} & \text{if} \; (\lambda, \pi) \; \text{is of bottom type.}
    \end{cases}
\end{equation}
With the above definition, $(\mathcal{R}_{R}, B^{R})$ forms an integral
cocycle. The corresponding cocycle over the Zorich transformation is called the Zorich cocycle and is denoted by $B^{Z}$. 

Let us consider a diagram with vertices the set of all combinatorial data in the Rauzy class $\mathfrak{R}$. We will denote it by $\Pi(\mathfrak{R})$ and connect two vertices with
an arrow labeled by $t$ or $b$ specifying the type of operation needed 
to go from the initial vertex of the edge to its terminal vertex. Now if $\gamma$ is a path obtained by concatenation of labeled edges $\gamma_1, \gamma_2,\dots, \gamma_k$, for some $k\in \mathbb{N}$, we let $B_{\gamma}$ be the product of the corresponding renormalization matrices. That is,
\begin{equation}
    B_{\gamma}:=B_{\gamma_k}\cdots B_{\gamma_1}.
\end{equation}
Note that if $\lambda^{(\gamma)}$ is the length vector corresponding to the induction of $(\lambda, \pi)$ following the path $\gamma$ (we are tacitly assuming that $\gamma$ starts at $\pi$) then 
\begin{equation}
    \lambda^{(\gamma)}B_{\gamma}=\lambda.
\end{equation}
where we are viewing $\lambda^{(\gamma)}$ and $\lambda$ as row vectors.

\subsection{Veech's zippered rectangles construction}
\label{VZR}

The zippered rectangles construction is an algorithm for constructing translation 
surfaces with translation flows using interval exchange transformations and locally constant roof functions $h$ that may be viewed as vectors in the positive 
cone $H^{+}(\pi)$ of a subspace $H(\pi) \subset \mathbb{R}^{d}$ of dimension $2g$ (where $g$ is the genus of the translation surface). For $h \in H^{+}(\pi)$, as far as the translation flow in the north direction is concerned, it is measurably isomorphic with the special flow with roof function $h$ over the base IET $T_{\lambda, \pi}$. 

Consider the linear transformation $\Omega_{\pi}$ given by the following formula (it is indeed the Poincar\'e duality intersection form, see \cite{viana2008dynamics}).

\begin{equation}
    (\Omega_{\pi})_{(\alpha, \beta)}:= 
    \begin{cases}
    +1 & \text{if} \; \pi_{b}(\beta)< \pi_b(\alpha), \; \pi_t(\beta)> \pi_t(\alpha)\\
    -1 & \text{if} \; \pi_b(\beta)> \pi_b(\alpha), \; \pi_t(\beta)< \pi_t(\alpha)\\
    0 & \text{otherwise}.
    
    \end{cases}
\end{equation}

We let 
\begin{equation}
    T^{+}(\pi):=\bigg\{\tau \in \mathbb{R}^{d}: \sum_{\pi_{t}(\alpha)< \pi_{t}(\beta)} \tau_{\alpha}>0, \; \sum_{\pi_b(\alpha)> \pi_{b}(\beta)} \tau_{\alpha}<0\bigg\}
\end{equation}
and define $H^{+}(\pi) \subset \mathbb{R}_{+}^{d}$ to be the image of $T^{+}(\pi)$ under the transformation $-\Omega_{\pi}$.

For $\tau \in T^{+}(\pi)$, we let $h=\Omega_{\pi}(\tau)$. The suspension surface corresponding to these data is constructed
by considering rectangles of the form $I_{\alpha}\times [0,h_\alpha]$ and applying certain identifications. The coordinates $\tau_\alpha$'s  of $\tau\in \mathbb R^d$ are used as zippers, that is, they determine the heights up to which two adjacent rectangles are 
sewn together in the vertical direction. The top edge of the rectangle of base $I_\alpha$ is identified with the subinterval $T_{\lambda, \pi}(I_\alpha)$, that is, with the subinterval of $I$ of label $\alpha$ in the partition of $I$ into the successive subintervals of labels $\pi^{-1}_{b}(1), \dots,\pi^{-1}_{b}(d)$. The final result is a surface $M(\lambda , \pi, \tau, h)$ equipped with charts whose transition maps are translations. For a detailed explanation of this construction, we refer the reader to \cite{10.2307/1971391}, or \cite{viana2008dynamics}.

It is well known that $N(\pi)$, the kernel of the linear transformation $\Omega_{\pi}$, has dimension $\kappa-1$, where $\kappa$ is the number of singularities (including the removable ones) of the suspension surface.

Moreover
\begin{equation}
    d=2g+\kappa -1.
\end{equation}

Veech \cite{10.2307/1971391} showed that the locally constant subbundles $H(\pi), N(\pi)$ are invariant and contravariant under the Rauzy--Veech cocycle, i.e.,
\begin{align}
    B^{R}(\lambda, \pi)\cdot H(\pi)=H(\pi') \\
    B^{R*}(\lambda, \pi)\cdot N(\pi')=N(\pi)
\end{align}
whenever $\mathcal{R}_{R}(\lambda, \pi)=(\lambda', \pi').$ It is also well known that if $\gamma$ is a loop starting and ending at $\pi$ the action of $B_{\gamma}^{*}$ sends a certain basis (consisting of vectors in $N(\pi) \cap \mathbb{Z}^{d}$) of $N(\pi)$ to itself, and thus acts as identity on $N(\pi)$, see \cite{10.2307/1971391}. 

We let $\theta_1(\mathfrak{R}) \geq \theta_2(\mathfrak{R}) \geq \dots \geq \theta_{2g}(\mathfrak{R})$ be the Lyapunov exponents of the Zorich cocycle restricted to $H(\pi)$. Forni \cite{MR1888794} showed that the Zorich cocycle restricted to $H(\pi)$ is non-uniformly hyperbolic, i.e., its Lyapunov exponents are all nonzero. This was later generalized by Avila and Viana \cite{MR2350698}, where they showed that the cocycle has simple Lyapunov spectrum. 

\begin{thm} (\cite{MR1888794}, \cite{MR2350698}) For any Rauzy class $\mathfrak{R} \subset \mathfrak{S}_d$ the Zorich cocycle on $\mathbb{P}_{+}^{d-1} \times \mathfrak{R} $ is non-uniformly hyperbolic and has simple Lyapunov spectrum. Therefore,

\begin{equation}
    \theta_1(\mathfrak{R})>\theta_2(\mathfrak{R})>\dots>\theta_g
    (\mathfrak{R})>0> -\theta_{g}(\mathfrak{R})>\dots>-\theta_{1}(\mathfrak{R}).
\end{equation}
The symmetry of the exponents is due to the cocycle being symplectic. 
\end{thm}

In the sequel, we do not need to use the full power of this theorem. Indeed, we only use a simpler result of \cite{MR1888794} that there are two positive Lyapunov exponents.

\subsection{Elimination for IETs of non-rotation type}

In this section we apply the abstract elimination 
Theorem~\ref{parexquan} to the
cocycles induced from the Rauzy--Veech cocycle.

Let $\gamma \in \Pi(\mathfrak{R})$ be a path in the Rauzy diagram that starts and 
ends at $\pi$ and such that the coefficients of 
the corresponding Rauzy matrix $B_{\gamma}$ are all positive. We then let $\Delta:= B_{\gamma}^{*}. \mathbb{P}_{+}^{d-1}$ (where
$B_{\gamma}^{*}$ acts projectively)  and consider the first return map of 
$\mathcal{R}_{R}$ to this simplex ($\Delta$ corresponds to the set of those parameters $(\lambda,\pi)\in \mathbb{P}_{+}^{d-1}$ whose renormalization follows the path $\gamma$ in the first $k$ steps, where $k$ is the length of $\gamma$).
The projection onto the first coordinate of the return map yields a map $T:\Delta \to \Delta$, whose domain of definition will be denoted by $\Delta^{1}$. By the Poincar\'e recurrence theorem, $\Delta^{1}$ has full measure inside $\Delta$. Proceeding in this manner, we may define $\Delta^{i}$ as the measurable subset of $\Delta$ consisting of all the points that return to $\Delta$ at least $i$ times. $\Delta^{\infty}$ is then the intersection of $\Delta^{i}$'s and represents the set of points that return infinitely many times. The set $\Delta^{\infty}$ has full measure in $\Delta$, but in fact, more is true. 

\begin{lemma} (Theorem 29, \cite{MR3522612})
\label{Hausret}
Let $\Delta$ be a simplex in $\mathbb{P}_{+}^{d-1}$ admitting a partition $\{\Delta^{(\ell)}\}_{\ell \in \mathbb{Z}}$, and let $T: \Delta \to \Delta$ be a map with bounded distortion such that for every $\ell \in \mathbb{Z}$, $T|_{\Delta^{(\ell)}}$ is a projective transformation. Assume that $T$ is fast decaying. Then
\begin{equation}
    \emph{HD}\left(\Delta \setminus \Delta^{\infty}\right)<d-1.
\end{equation}
\end{lemma}
Where the $\Delta^{(\ell)}$'s in this setting correspond to the sets $\Delta_{\gamma^{(\ell)}}$'s, for $\gamma^{(\ell)}$'s in $\Pi(\mathfrak{R})$ that start and end at $\pi$, which correspond to IETs whose renormalization combinatorics follows the path $\gamma \gamma^{(\ell)}$ before returning to $\Delta$.
Hence $\Delta^{\infty}$, which is where one can apply $T$ infinitely many times, has positive Hausdorff codimension.  We let $A^{(\ell)}=A|_{\Delta^{(\ell)}}:=B_{\gamma \gamma^{(\ell)}}$ and consider the cocycle $(T,A)$. Avila and Leguil showed that $(T,A)$ is a fast decaying cocycle, see \cite{MR3933880}, section $3$. We consider an IET $T_{\lambda, \pi}$ as the first return map of the suspension flow with parameters $(\lambda, \pi, h)$ for $h=(1,\dots,1) \in \mathbb{Z}^{d}$. We use the fact that there are two positive Lyapunov exponents to show that a typical line not passing through the origin whose direction lies in $\Theta:=\mathbb{P}_{+}^{d-1}$ does not intersect the center stable subspace of a typical IET $T_{\lambda, \pi}$. For non-rotation class IETs, i.e., when $g \geq 2$ we have the following
\begin{thm}
(Theorem 5.1, \cite{MR2299743})
Let $L$ be a line in the direction of $\Theta=\mathbb{P}_{+}^{d-1}$ not passing through the origin. Then for almost every $\lambda$ we have
\begin{equation}
   L \cap E^{cs}(\lambda, \pi)= \emptyset\,.
\end{equation}
\end{thm}
\begin{proof}
Note that we need to show the statement only for one pair $\pi=(\pi_t, \pi_b)$ in every Rauzy class and thus we may assume that $\pi$ is a standard pair, that is
a pair such that $\pi_b \circ \pi_t^{-1}(d)=1$ and $\pi_b \circ \pi_t^{-1}(1)=d$  (such a $\pi$ exists in every Rauzy class, see \cite{viana2008dynamics}, Prop. 1.24). Assume that the statement does not hold. Therefore there exists a positive measure set $E$ of $\lambda \in \Delta$ and a line $L$ in the direction of $\Theta$ not passing through the origin such that 
\begin{equation}
    L \cap E^{cs}(\lambda, \pi) \neq \emptyset \,, \quad \text{ for all }\lambda \in E\,.
\end{equation}
 Let $L=\{h^{(1)}+th^{(2)}| t\in \mathbb{R}\}$ with $h^{(1)} \neq 0$ and $h^{(2)} \in \mathbb{R}_{+}^{d}$ (we only need the components of $h^{(2)}$ to be nonzero). Pick $\lambda \in E$ a density point which happens to be Oseledets regular as well. Note that 
\begin{equation}
    \left(A_{n}(\lambda, \pi) \cdot L\right) \cap E^{cs}(\mathcal{R}^{n}_{R}(\lambda, \pi))=
    A_{n}(\lambda, \pi) \cdot (L \cap E^{cs}(\lambda, \pi)) \neq \emptyset
\end{equation}
and thus, if $A_\alpha$, for $\alpha \in \mathcal A$, are some renormalization matrices, we get that there exists a positive measure set of $\lambda \in \Delta$ such that, for all $\alpha \in \mathcal A$,
\begin{equation}
    A_{\alpha} \cdot L \cap E^{cs}(\lambda, \pi) \neq \emptyset.
\end{equation}
We define $v_\beta=(v^{\alpha}_\beta)_{\alpha\in \mathcal A}$, for $\beta \in \mathcal A$, by 
\begin{equation}
    v_{\beta}^{\alpha}:=\begin{cases}
    -1 & \text{if} \; \pi_b(\beta)<\pi_b(\alpha),\; \pi_t(\beta)>\pi_t(\alpha) ,\\
    1& \text{if} \; \pi_b(\beta)>\pi_b(\alpha),\; \pi_t(\beta)<\pi_t(\alpha),\\
    0 & \text{otherwise.} 
    \end{cases}
\end{equation}
These are the columns of the matrix $\Omega_{\pi}$ that we defined at the beginning of section \ref{VZR} and so their span is exactly the subspace $H(\pi)$. We now construct  the paths $\gamma_\alpha$ in the Rauzy diagram such that the corresponding renormalization matrices take certain special forms. For $\alpha \in \mathcal A\setminus \{\alpha_t\}$, we let $\gamma_\alpha$ be the path obtained in the following manner. We let $\alpha_t$ win until it competes against $\alpha$. Then we let $\alpha$ win until it competes with $\alpha_t$ again. Then we let $\alpha_t$ win until it competes with $\alpha_b$ again. We call the Rauzy path obtained in this manner $\gamma_\alpha$ and we let $A_{\alpha}:=B_{\gamma_\alpha}$. We let $\gamma_{\alpha_t}$ be the path obtained via letting $\alpha_t$ win against every other element until it competes with $\alpha_b$ again. It can be easily seen that
\begin{equation}
    A_{\alpha}\cdot (z_\beta)_{\beta\in \mathcal A}
    = (z_\beta)_{\beta\in \mathcal A}
    +z_\alpha (v_\alpha-v_{\alpha_t})-z_{\alpha_t} v_{\alpha_t} 
    \;\;\;  \alpha \in \mathcal A\setminus \{\alpha_t\},
\end{equation}
and 
\begin{equation}
 A_{\alpha_t} \cdot (z_\beta)_{\beta\in \mathcal A}
    =
    (z_\beta)_{\beta\in \mathcal A}
    -z_{\alpha_t}v_{\alpha_t}. 
\end{equation}

It is easy to see that for almost every $\lambda$, $E^{cs}(\lambda, \pi)$ is perpendicular to $\lambda$ (it follows from the fact that $\big\langle (A_n(\lambda, \pi)^{-1})^{*}\lambda, A_{n}(\lambda, \pi)\delta\big\rangle=\big\langle \lambda, \delta \big\rangle$ converges to zero as the left hand-side decays exponentially fast and $\delta \in E^{cs}$). Therefore,
\begin{equation}
    h^{(1)}-\frac{\big\langle \lambda, h^{(1)} \big\rangle }{\big\langle \lambda, h^{(2)} \big\rangle} h^{(2)} \in E^{cs}(\lambda, \pi),
\end{equation}
similarly we get, for $\alpha \in \mathcal A\setminus\{\alpha_t\}$, 
\begin{equation} 
\begin{split}
h^{(1)}&+h^{(1)}_\alpha (v_{\alpha}-v_{\alpha_t})-h^{(1)}_{\alpha_t} v_{\alpha_t} \\&- \frac{\big\langle \lambda, h^{(1)}+h^{(1)}_{\alpha} (v_{\alpha}-v_{\alpha_t}) -h^{(1)}_{\alpha_t} v_{\alpha_t} \big\rangle}{\big\langle \lambda, h^{(2)}+h^{(2)}_\alpha(v_\alpha-v_{\alpha_t})-h^{(2)}_{\alpha_t} v_{\alpha_t} \big\rangle}(h^{(2)}+h^{(2)}_\alpha (v_\alpha-v_{\alpha_t})-h^{(2)}_{\alpha_t} v_{\alpha_t}) \in E^{cs}(\lambda, \pi)
\end{split}
\end{equation}
and
\begin{equation}
    h^{(1)}-h^{(1)}_{\alpha_t} v_{\alpha_t}- \frac{\big\langle \lambda, h^{(1)}-h^{(1)}_{\alpha_t}v_{\alpha_t} \big\rangle}{\big\langle \lambda, h^{(2)}-h^{(2)}_{\alpha_t}v_{\alpha_t} \big\rangle}(h^{(2)}-h^{(2)}_{\alpha_t}v_{\alpha_t}) \in E^{cs}(\lambda, \pi).
\end{equation}

For typical $\lambda$ the coefficient of $v_\alpha$ in the $\alpha$-labeled relation is nonzero and therefore we can conclude that 
\begin{equation}
    v_\alpha \in E^{cs}(\lambda, \pi)+\big\langle h^{(2)} \big\rangle \,, \; \; \quad  \forall \alpha \in{\mathcal A}\,,
\end{equation}
for almost every $\lambda \in E$. Hence $H(\pi) \subset E^{cs}(\lambda, \pi)+\big\langle h^{(2)} \big\rangle$ and therefore $d-2 \geq \dim(E^{cs}(\lambda, \pi) \geq d-1$, which is a contradiction.
\end{proof}

Applying Theorem~\ref{parexquan} to the cocycle $(T,A)$ (the inducing of the Rauzy--Veech cocycle to $\Delta$ as explained in the beginning of this section) with $\Theta=\mathbb{P}_{+}^{d-1}$ and using Lemma~\ref{Hausret}, we obtain the following

\begin{thm} 
\label{IETsp}
There exist $\epsilon>0$ and a measurable subset $\Delta_{sp} \subset \Delta$, whose complement does not have positive Hausdorff codimension, such that the following holds. For all $\lambda \in \Delta_{sp}$ and all $t \in (0,1)$, there exists $n(t,\lambda)$ so that $\forall n \geq n(t,\lambda)$ we have 
\begin{equation}
     \frac{\#\left\{1\leq i \leq n: \|A_i(\lambda, \pi) \cdot th\|_{\mathbb{R}^{d}/
    \mathbb{Z}^{d}}>\epsilon\right\}}{n}>\epsilon\,.
\end{equation}
Moreover $n(t,\lambda)$ may be chosen to be of the form 
$C\max\{-\ln(t),-\ln(1-t)\}+K(\lambda)$ for some constant $C>0$ (not depending on $\lambda$) and a function $K(\lambda)>0$.
\end{thm}

\section{Twisted integrals}
\label{sec:twist_int}
In this section, following Bufetov and Solomyak \cite{bufetov2014modulus}, \cite{MR3773061}, \cite{2019arXiv190809347B}, we prove upper bounds on twisted ergodic integrals for suspension flows over substitution systems.

\subsection{Substitutions}
We briefly discuss the notion of substitutions to the extent that 
is relevant to this paper. For a general text on substitutions, see \cite{fogg}. Let $\mathcal{A}\equiv\{1,2,\dots,d\}$ be a 
finite alphabet and $\mathcal{A}^{+}$, $\mathcal{A}^{\mathbb{N}}$ be the set of all finite and infinite words with letters in 
$\mathcal{A}$, respectively. A substitution is a map $\zeta: \mathcal{A} \to \mathcal{A}^{+}$, which is extended to an action on $\mathcal{A}^{+}, 
\mathcal{A}^{\mathbb{N}}$ via concatenation. To each 
substitution $\zeta$, there corresponds a matrix called the \textit{substitution matrix} $S_{\zeta}$ as follows

\begin{equation}
    S_{\zeta}(\alpha,\beta):=\text{number of symbols}\; \alpha \; \text{in}\; \zeta(\beta).
\end{equation}

We denote by $\mathfrak{A}$ a set of substitutions 
$\zeta$ for which all letters appear in the set 
$\{\zeta(\alpha), \alpha \in \mathcal{A}\}$ and there exists $\alpha \in
\mathcal{A}$ with $|\zeta(\alpha)|>1$, where for a word $v\in \mathcal{A}^{+}$, $|v|$ denotes its length. 

We let $\Omega$ be the set of 1-sided infinite sequences of substitutions belonging to $\mathfrak{A}$. In the following subsection, we show that by 
choosing $\mathfrak{A}$ in an appropriate manner, dictated by 
the Rauzy--Veech induction, we can identify the space of minimal uniquely ergodic interval exchange transformations with the set 
$\Omega$. Under this identification the Rauzy--Veech 
renormalization cocycle corresponds to the renormalization cocycle in the S-adic framework.

We will denote by $\mathrm{a}=\{\zeta_j\}_{j=1}^{\infty}$ an element of $\Omega$. We define $X_{\mathrm{a}}$ the substitution space corresponding to $a$ as the 
set of all bi-infinite sequences of letters (of $\mathcal{A}$) whose every finite sub-word is a sub-word of a substitution word of
the form $\zeta_1 \circ \dots\circ \zeta_{k} (\beta)$, for some $k\in \mathbb{N}$ and $\beta \in \mathcal{A}$. The (topological) space $X_\mathrm{a}$ equipped with the natural shift $\sigma$ is a dynamical system $(X_\mathrm{a}, 
\sigma)$. Whenever this system is uniquely ergodic, its unique invariant measure will be denoted by 
$\mu_{\mathrm{a}}$. When there is no ambiguity, we may drop the subscript. The renormalization cocycle 
is then a linear cocycle on $\mathbb{R}^{d}$ over the shift map on $\Omega$ (not to be confused with the 
shift on $X_{\mathrm{a}}$), given by the transposition of the 
substitution matrix: for all $a\in \Omega$ and $n\in \mathbb N$,

\begin{equation}
    \mathbb{A}(\mathrm{a}):=S_{\zeta_1}^{t}, \; \quad  \mathbb{A}(n,\mathrm{a})=\mathbb{A}(\sigma^{n-1}(\mathrm{a}))\cdot \cdot \cdot \mathbb{A}(\mathrm{a}). 
\end{equation}

For substitutions $\zeta_{1}, \zeta_{2},\dots,\zeta_{k}$ their composition is defined as compositions of maps from $\mathcal{A}^{+}$ to itself. We will use the following notations
\begin{equation}
    \zeta^{[k]}:=\zeta_1 \circ \zeta_2 \circ \cdots\circ \zeta_{k},
\end{equation}
\begin{equation}
    S_{k}:=S_{\zeta_{k}},
\end{equation}
\begin{equation}
    S^{[k]}:=S_{\zeta^{[k]}}=S_{\zeta_1} \circ S_{\zeta_{2}} \circ \cdots\circ S_{\zeta_{k}}.
\end{equation}

For a substitution system $(X_\mathrm{a}, \sigma)$ and a vector $\vec{s} \in \mathbb{R}_{+}^{d}$ we denote the special flow with roof function $\vec{s}$ (the roof for a symbol $\alpha \in \mathcal{A}$ is the corresponding component $s_{\alpha}$ of $\vec{s}$) by $(\mathfrak{X}_{\mathrm{a}}^{\vec{s}}, h_{t})$. The first return map of this flow to the transverse section $X_\mathrm{a}$ is given by $\sigma$. 

\subsection{Symbolic representation of IETs}
It is well known that a uniquely ergodic IET can be uniquely determined via its sequence of Rauzy matrices. For parameters $(\lambda, \pi)$ of top type corresponding to a uniquely ergodic IET, we define the substitution $\zeta_{1}$ via the following formula
\begin{equation}
    \zeta_{1}(\alpha):=
    \begin{cases}
     \alpha_{b} \alpha_{t} & \text{if} \; \alpha= \alpha_b\\
     \alpha & \text{otherwise.} 
    \end{cases}
\end{equation}
For $(\lambda, \pi)$ of bottom type we let
\begin{equation}
  \zeta_{1}(\alpha):= \begin{cases}
  \alpha_{b} \cdot \alpha_{t} & \text{if} \; \alpha= \alpha_{t}\\
  \alpha & \text{otherwise,}
  \end{cases}
\end{equation}
where the position of the dot distinguishes the past and future. Proceeding in this manner, we get a sequence $\mathrm{a}=\{\zeta_{j}\}_{j=1}^{\infty}$ corresponding to $T_{\lambda, \pi}$.
Under this identification, every uniquely ergodic IET yields a sequence of substitutions on the alphabet $\mathcal{A}$, given by the set of labels of the intervals. Then the IET corresponds to the 
substitution shift space $X_\mathrm{a}$ given by this sequence of substitutions. If $T_{\lambda, \pi}$ is uniquely ergodic, then the Lebesgue measure on the interval corresponds to $\mu_{\mathrm{a}}$, the unique ergodic measure of $(X_{\mathrm{a}}, \sigma)$. It is now easy to see that in this setting the Rauzy cocycle corresponds to the renormalization cocycle of the substitution system. Below, we consider the returns of Rauzy 
induction to a certain set defined by a word $\mathbf{q}$ that satisfies certain good properties. For the existence of such a word (or correspondingly such a loop in the Rauzy diagram) we refer the reader to section 7 of \cite{2019arXiv190809347B}.

\subsection{Spectral measures}
We will briefly give the definition of spectral measures for flows and discrete-time dynamical systems. For more on spectral measures, see \cite{parry2004topics} or
\cite{KATOK2006649}.

Let $(X, T, \mu)$ be a measure preserving dynamical system and $f \in L^{2}(X,\mu)$ an observable. Denote by $U_{T}$ the Koopman unitary operator on $L^{2}(X,\mu)$
associated with post-composition with $T$. 

A theorem of Wiener then implies that there exists a unique measure $\sigma_f$, called the spectral measure of $f$, that satisfies

\begin{equation}
  \widehat{\sigma}_{f}(-n)= \int_{0}^{1} e^{2\pi\imath n\omega} d\sigma_{f}(\omega) = \big\langle U_{T}^{n}f,f \big\rangle.
\end{equation}

Similarly for a measure preserving flow $(Y,h_{t}, \mu)$ and a test function $f\in L^{2}(Y, \mu)$, $\sigma_{f}$ is the unique finite positive measure on $\mathbb{R}$ such that
\begin{equation}
    \widehat{\sigma}_{f}(-\tau)= \int_{-\infty}^{\infty} e^{2\pi\imath \tau t} d\sigma_{f}(t) = \big\langle f \circ h_{\tau}, f  \big\rangle.
\end{equation}
The total mass of $\sigma_f$ is 
$\|f\|_{2}^{2}$.

\begin{definition}
The {\em twisted Birkhoff integral} of $f$ (with respect to $\omega$, the twist factor) for a point $y \in Y$, up to time $R>0$, is given by

\begin{equation}
    S_{R}^{(y)}(f, \omega)= \int_{0}^{R} e^{-2\pi\imath  \omega t} f \circ h_{t}(y)dt.
\end{equation}
\end{definition}
\begin{rmk}
Twisted Birkhoff sums are defined similarly for discrete-time dynamical systems.
\end{rmk}

There is a close connection between polynomial bounds for twisted Birkhoff integrals (sums) and H\"older bounds for the associated spectral measures as made clear in the next standard lemma.

\begin{lemma} 
\label{Twisted Birkhoff Spectral}
 Suppose that for some $\omega \in \mathbb{R}$, $R_0 >0$, and $\alpha \in (0,1)$ we have
\begin{equation}
    \left\|S_{R}^{(y)}(f, \omega)\right\|_{L^{2}(Y,\mu)} \leq C_1 R^{\alpha}\,,  \quad \forall R \geq R_0.
\end{equation}
Then
\begin{equation}
    \sigma_f\left([\omega -r, \omega +r]\right) \leq \pi^{2} 2^{-2\alpha} C_1^{2}r^{2(1-\alpha)}\,, \;\; \quad \forall r \leq (2R_0)^{-1}.
\end{equation}
\end{lemma}

\begin{proof}
Note that
\begin{equation}
\label{birks}
\begin{split}
    \left\|S_{R}^{(y)}(f, \omega)\right\|^{2}_{L^{2}} &= \int_{Y}\int_{0}^{R}e^{-2\pi\imath  \omega t} f\circ h_{t}(y)dt \overline{\int_{0}^{R}e^{-2\pi\imath  \omega s} f \circ h_{s}(y)ds}\;d\mu(y)\\ 
    & =\int_{0}^{R}\int_{0}^{R} e^{2\pi\imath  \omega (s-t)} \big\langle f \circ h_{t}, f \circ h_s \big\rangle dt ds \\
    &= \int_{-R}^{R} (R-|\ell|) e^{2\pi\imath  \omega \ell} \widehat{\sigma}_f(\ell) d\ell= \int_{-R}^{R}(R-|\ell|)\int_{-\infty}^{\infty} e^{2\pi\imath  \ell (\omega -\theta)} d\sigma_f(\theta) d \ell\\
    &=\int_{-\infty}^{\infty}\int_{-R}^{R} (R-\left|\ell\right|) e^{2\pi\imath \ell(\omega -\theta)} d\ell d\sigma_{f}(\theta)
\end{split}
\end{equation}
We now use the Fej\'er kernel to simplify the last integral
\begin{equation}
\label{fe}
    \int_{-R}^{R}\left(R-|\ell|\right)e^{2\pi\imath \ell \xi} d\ell=\int_{0}^{R}\int_{-\ell}^{\ell} e^{2\pi\imath  s \xi } ds d\ell= \left(\frac{\sin(\pi R \xi)}{\pi \xi}\right)^{2}.
\end{equation}

Therefore substituting \eqref{fe} in \eqref{birks} for $\xi:=\omega -\theta$ we get that for $r=1/2R\leq1/2R_{0}$
\begin{equation}
C_1^{2}R^{2\alpha} \geq \int_{-\infty}^{\infty}    \left(\frac{\sin(\pi R (\omega-\theta))}{\pi (\omega-\theta)}\right)^{2} d\sigma_f(\theta) \geq \int_{\omega-r}^{\omega+r} \frac{4R^{2}}{\pi^{2}}d\sigma_f(\theta)
\end{equation}
which yields the desired result. Note that we have used the fact that $\big|\frac{\sin(\xi)}{\xi}\big| \geq \frac{2}{\pi}$
for all $|\xi|<\frac{\pi}{2}$.
\end{proof}

\subsection{Cylindrical functions}
In general, for measurable functions the rate of decay of correlations may be very slow (\cite{petersen}). 
Therefore, it is not expected that H\"older bounds hold for spectral measures associated to general measurable observables. Thus, it 
is reasonable to restrict our classes of observables to subclasses of integrable functions with some degrees of regularity. In this paper, we will work with the class of weakly Lipschitz 
functions. These functions are good observables for detecting divergence of orbits.
\\

\noindent
For $\ell \geq 1$ we have the following canonical decomposition (disjoint in measure)

\begin{equation}
    \mathfrak{X}^{\vec{s}}=\bigcup_{\alpha \in \mathcal{A}}\mathfrak{X}_{\alpha}^{(\ell)}:= \bigcup_{\alpha \in \mathcal{A}} \zeta^{[\ell]}[\alpha] \times [0,s_{\alpha}^{(\ell)}]
\end{equation}
where the $\zeta^{[\ell]}[\alpha]$'s (which are elements whose symbolic word starts with $\zeta^{[\ell]}(\alpha)$) correspond to the subintervals of the base in the $\ell$'th step of Rauzy renormalization, and $s^{\ell}_{\alpha}$'s are the heights of the corresponding towers. 
This implies that for $(x,t) \in \mathfrak{X}^{\vec{s}}_{\mathrm{a}}$ there exist $\alpha \in \mathcal{A}, x' \in \zeta^{[\ell]}[\alpha]$ and $t' \in [0,s_{\alpha}^{(\ell)}]$ such that $h_{t'}(x',0)=(x,t).$

\begin{definition}
A function $f: \mathfrak{X}^{\vec{s}}_{\mathrm{a}} \to \mathbb{C}$ is a {\em bounded cylindrical function} of level $\ell$ (for an integer $\ell \geq 0$) if, for all $(x,t) \in \mathfrak{X}^{\vec{s}}_{\mathrm{a}}$ we have
\begin{equation}
    f(x,t):= \sum_{\alpha \in \mathcal{A}} \mathbbm{1}_{\zeta^{[\ell]}[\alpha]}(x)\cdot \psi_{\alpha}^{(\ell)}(t) ,\;\quad \text{with} \quad  \;\psi_{\alpha}^{(\ell)} \in L^{\infty}[0,s_{\alpha}^{(\ell)}].
\end{equation}
\end{definition}

\begin{definition}
A bounded function $f: \mathfrak{X}^{\vec{s}}_{\mathrm{a}} \to \mathbb{C}$ is {\em weakly Lipschitz}, and we write $f \in Lip_{w}(\mathfrak{X}^{\vec{s}}_{\mathrm{a}})$, if there exists a constant $C>0$ such that, for all $\ell \geq 0$ and all $\alpha \in \mathcal{A}$,  for every $x,y \in \zeta^{[\ell]}[\alpha]$ and $t \in [0,s^{(\ell)}_\alpha]$, we have

\begin{equation}
\label{Lip}
\left|f(x,t)-f(y,t)\right| \leq C \mu(\zeta^{[\ell]}[\alpha])\,.
\end{equation}
The norm of a function in $Lip_{w}(\mathfrak{X}^{\vec{s}}_{\mathrm{a}})$ is defined as follows

\begin{equation}
\|f\|_{L}:=\|f\|_{\infty}+\tilde{C}\,,
\end{equation}
where $\|.\|_{\infty}$ denotes the $L^{\infty}$-norm and $\tilde{C}>0$ is the smallest positive real number satisfying \eqref{Lip}.

\end{definition}

\begin{definition}
\label{def:good_ret_word}
A word $v \in \mathcal{A}^{+}$ is called a {\em good return} word for a substitution $\zeta$ if $v$ starts with some letter $c$ and $vc$ occurs in the substitution $\zeta(\alpha)$ of every letter $\alpha\in \mathcal A$. We denote by $GR(\zeta)$ the set of good return words for $\zeta$.
\end{definition}

A word $\mathbf{q}=q_{1}\dots q_{m} \in \mathfrak{A}^{m}$ is called simple if no two occurrences of it can overlap, that is, for every $1 \leq i \leq m$, $q_{1}\dots q_{i} \neq q_{m-i+1}\dots q_{m}$. Then, we will compose the subsequent substitutions and write $\zeta(\mathbf{q}):= \zeta(q_{1})\dots\zeta(q_{m}).$ 
\begin{definition}
\label{def:pop_vec}
The {\em population vector} $\vec{\ell}(v) \in \mathbb{Z}^{d}$ corresponding to $v=v_{1}v_{2}\cdots v_{k}$ is the vector whose $\alpha$-th coordinate is the number of integers $1 \leq i \leq k$ for which $v_{i}=\alpha$.
\end{definition}

In the sequel, we will always assume 
that $\mathbf{q}$ is a fixed simple word with the property that $\{\vec{\ell}(v): v \in GR(\zeta) \}$, where $\zeta= \zeta(\mathbf{q})$, generates $\mathbb{Z}^{d}$ as an Abelian group. These return words will later provide us with a pseudo-norm equivalent to $\|.\|_{\mathbb{R}^{d}/\mathbb{Z}^{d}}$. 

\smallskip
We let $\mathbf{q.q}$ be the word obtained by concatenation of $\mathbf{q}$ 
to itself. For the existence of such a word $\mathbf{q}$ in the Rauzy--Veech setting we refer to \cite{2019arXiv190809347B}, Lemma~7.1. Throughout the rest of the paper, we will assume that $\theta_1$ is the largest Lyapunov exponent of the cocycle obtained by inducing on the set defined by the word $\mathbf{q.q}$.
\begin{definition}
For $\mathbf{q}$ as above we define 
\begin{equation}
    \Omega_{\mathbf{q}}:=\left\{\mathrm{a}\in \Omega: \mathrm{a} \, \,\text{starts with}\, \,\,\mathbf{q}.\mathbf{q}\right\}
\end{equation}
\end{definition}
We denote the length of $\mathbf{q}$ by $m$. If $\mathrm{a'}\in \Omega_{\mathbf{q}}$ then there exist $p_{i}\in \mathfrak{A}^{+}$, depending on $\mathrm{a'}$, such that
\begin{equation}
    \mathrm{a}:= \sigma^{m}(\mathrm{a'})=\mathbf{q}p_{1}\mathbf{q}\mathbf{q}p_{2}\mathbf{q}\mathbf{q}\dots
\end{equation}
Therefore we can simply rewrite $\mathrm{a}$ as $\{\mathrm{a}_j\}_{j=1}^{\infty}$ where
\begin{equation}
    \mathrm{a}_{j}= \mathbf{q}\mathrm{p_j}\mathbf{q},\; \quad \textit{\rm for some}\; p_{j} \in \mathfrak{A}^{+}\,.
\end{equation}

\begin{definition}
\label{defi}
We say that $\mathrm{a}\in \Omega'_{\mathbf{q}}$ if there exists $\mathrm{a'}\in \Omega_{\mathbf{q}}$ such that 
\begin{itemize}
    \item $\mathrm{a}=\sigma^{m}(\mathrm{a'})$
    \item $\exists \ell_{0}(\mathrm{a}) \in \mathbb{N}: $
    \begin{equation}
        \left|\frac{\log\|S^{[\ell]}\|_{1}}{\ell}-\theta_1 \right|\leq \theta_1/8, \quad  \forall \ell \geq \ell_{0}(\mathrm{a}).
    \end{equation}
\end{itemize}
\end{definition}
We recall that $\theta_1$ is the largest Lyapunov exponent of the cocycle induced by the Rauzy--Veech cocycle on $\Omega({\mathbf q})$.

\subsection{Exponential sums}

For every word $v\in \mathcal{A}^{+}$, let $\vec{\ell}(v)$ denote the population vector of $v$, introduced in Definition~\ref{def:pop_vec}.

\begin{definition} For $\vec{s} \in \mathbb{R}_{+}^{d}$ the {\em tiling length} of a word $v=v_{1}v_{2}\cdots  v_{k}$, where $v_{i} \in \mathcal {A}$, is denoted by $|v|_{\vec{s}}$ and defined to be $\big\langle \vec{\ell}(v), \vec{s} \big\rangle$.
\end{definition}
For a word $v=v_1\dots v_k \in \mathcal{A}^{+}$, $\vec{s} \in \mathbb{R}_{+}^{d}$, and a letter $\alpha \in \mathcal{A}$ we define
\begin{equation}
    \Phi_{\alpha}^{\vec{s}}(v, \omega):=\sum_{j=1}^{k} \delta_{v_{j}, \alpha} \exp (-2\pi \imath \omega |v_{1}\dots v_{j-1}|_{\vec{s}})\,.
\end{equation}
Following \cite{10.2307/2374396}, for a matrix A with non-negative entries we define
\begin{equation}
     col(A)= \underset{i,j,k}{\max}
     \frac{A_{ij}}{A_{kj}}
\end{equation}
and observe that for a matrix $B$ with nonnegative entries we have
\begin{equation}
\label{col}
    col(AB) \leq col(A).
\end{equation}
Following \cite{MR3773061}, we let $\Pi_{n}^{\vec{s}}(\omega)$ be the complex matrix defined as follows
\begin{equation}
    \Pi_{n}^{\vec{s}}(\omega)(\beta, \alpha):= \Phi^{\vec{s}}_{\alpha}(\zeta^{[n]}(\beta), \omega) \,, \quad \text{\rm for all } \alpha, \beta \in \mathcal A.
\end{equation}
We assume that $\zeta_{n}(\beta)=u^{n, \beta}_1\dots u_{k}^{n, \beta}$, where $k=|\zeta_{n}(\beta)|$. Then,
\begin{equation}
\label{spectcoc}
\begin{split}
      \Pi_{n}^{\vec{s}}(\omega)( \beta, \alpha)&=\Phi^{\vec{s}}_{\alpha}(\zeta^{[n-1]}(u^{n, \beta}_1\dots u^{n, \beta}_{k}), \omega)\\
      &=\sum_{\gamma \in \mathcal{A}} \bigg(\sum_{u_{j}^{n, \beta}=\gamma} \exp(-2\pi\imath \omega |\zeta^{[n-1]}(u_{1}^{n,\beta}\dots u_{j-1}^{n,\beta})|_{\vec{s}})\bigg) \Phi_{\alpha}^{\vec{s}}(\zeta^{[n-1]}(\gamma), \omega)\\
      &=\sum_{\gamma \in \mathcal{A}} \bigg(\sum_{u_{j}^{n, \beta}=\gamma} \exp(-2\pi\imath \omega|u_{1}^{n, \beta}\dots u_{j-1}^{n, \beta}|_{S^{t}_{\zeta^{[n-1]}}\vec{s}}) \bigg) \Pi_{n-1}^{\vec{s}}(\omega)(\gamma, \alpha)
\end{split}
\end{equation}
which motivates the following definition. For two substitutions $\zeta, \xi$ we let,
\begin{equation}
\begin{split}
M_{\xi, \zeta}^{\vec{s}}(\omega)(\beta, \gamma)&:= \sum_{j \leq |\zeta(\beta)|, \; u_{j}^{\beta}= \gamma} \exp(-2\pi\imath \omega |\xi(u_{1}^{\beta}\dots u^{\beta}_{j-1})|_{\vec{s}}) \\ &=\sum_{j \leq |\zeta(\beta)|, u_{j}^{\beta} =\gamma} \exp (-2 \pi \imath\omega |u_1^{\beta}\dots u_{j-1}^{\beta}|_{S^{t}_{\xi}\vec{s}})
\end{split}
\end{equation}
where we assume that $\zeta(\beta)=u^{\beta}_1\dots u^{\beta}_{|\zeta(\beta)|}$. The above formula may be simply rewritten as
\begin{equation}
\label{coc}
    M^{\vec{s}}_{\xi, \zeta}(\omega)(\beta, \gamma)= \Phi^{S^{t}_{\xi}\vec{s}}_{\gamma}(\zeta(\beta), \omega).
\end{equation}
It is then rather straightforward to see that 
\begin{equation}
\label{M_id}
    M^{\vec{s}}_{\zeta_1, \zeta_2\zeta_3}(\omega)= M^{\vec{s}}_{\zeta_1 \zeta_2, \zeta_3} (\omega) \,M^{\vec{s}}_{\zeta_1, \zeta_2}(\omega)\,.
\end{equation}
We can now rewrite \eqref{spectcoc} as
\begin{equation}
    \Pi_{n}^{\vec{s}}(\omega)(\beta, \alpha)= \sum_{\gamma \in \mathcal{A}} M^{\vec{s}}_{\zeta^{[n-1]}, \zeta_n}(\omega)(\beta, \gamma) \Pi_{n-1}^{\vec{s}}(\omega)(\gamma, \alpha)
\end{equation}
and thus, by letting $M_{n}^{\vec{s}}(\omega):=M^{\vec{s}}_{\zeta^{[n-1]}, \zeta_n}(\omega)$ and observing that by definition $\Pi_{0}^{\vec{s}}(\omega)=Id$, we derive the following
\begin{equation}
    \Pi_{n}^{\vec{s}}(\omega)=M_{n}^{\vec{s}}(\omega)\Pi_{n-1}^{\vec{s}}(\omega)=
    \dots =M^{\vec{s}}_{n}(\omega)\cdots M_{1}^{\vec{s}}(\omega).
\end{equation}
In view of \eqref{coc}, it is easy to see that $M^{\vec{s}}(\omega)$ form a cocycle over the toral Rauzy--Veech cocycle. The following proposition provides the main tool  for estimating the twisted Birkhoff sums of weakly Lipschitz functions. We reproduce below the proof from~\cite{MR3773061} for completeness.

\begin{prop} (Proposition 3.4, \cite{MR3773061})
\label{mainbound}
Let $\mathrm{a}\in \Omega$ be a one-sided substitution sequence, 
and let $\zeta_{j}$ be the corresponding sequence of substitutions. Suppose that there exists a substitution $\zeta$ with a
nonempty set of good return words, such that $Q=S_\zeta$ is strictly 
positive and $\zeta_{j} = \zeta \xi_j \zeta$ for some substitution $\xi_{j}$ for 
all $j \geq 1$. Then there exists $c_1 \in (0,1)$, depending only on the substitution $\zeta$, such that for all $a,b \in \mathcal{A}$, $N 
\in \mathbb{N}$, and $\omega \in [0,1]$ we have
\begin{equation}
    \Phi_{\alpha}^{\vec{s}}(\zeta^{[N]}(\beta), \omega) \leq \|S^{[N]}\|_{1} 
    \prod_{n \leq N-1}\left(1-c_1 \cdot \underset{v\in GR(\zeta)}{\max}\|\omega|\zeta^{[n]}(v)|_{\vec{s}} \|_{\mathbb{R}/\mathbb{Z}}^{2}\right)
\end{equation}
In fact, we can take
\begin{equation}
    c_{1}:=\left(2d \cdot (\max_{\alpha, \beta} Q_{\alpha, \beta}) col(Q^{t})\right)^{-1}.
\end{equation}
\end{prop}
\begin{proof} For a matrix $A$, we denote $\vert A \vert$ the matrix whose entries are the absolute values of the entries of $A$. For matrices $A$ and $B$ we say that $A \leq B$ if $B-A$ has nonnegative entries. It is then easy to see that 
\begin{equation}
    \left| AB \right|  \leq \left| A \right| \left| B \right|.
\end{equation}
Thus we have
\begin{equation}
    \left| \Pi_n^{\vec{s}}(\omega) e_{\alpha}\right| \leq \left| \Pi_{n}^{\vec{s}}(\omega)\right| \vec{1} \leq \left| M^{\vec{s}}_{n}(\omega)\right| \cdots \left| M^{\vec{s}}_{1}(\omega)\right| \vec{1}
\end{equation}
where $\vec{1}=(1,1,\dots,1)^{t}$. For the rest of this proof we sometimes drop the superscript $\vec{s}$ and $\omega$ for simplicity. Since by assumption $\zeta_n= \zeta\xi_n\zeta$, by
formula \eqref{M_id} we have
\begin{equation}
\begin{split}
    M_n= M^{\vec{s}}_{\zeta^{[n-1]}, \zeta_n}(\omega)&=M^{\vec{s}}_{\zeta^{[n-1]}, \zeta \xi_n \zeta}(\omega)\\
    &=M^{\vec{s}}_{\zeta^{[n-1]} \zeta \xi_n, \zeta}(\omega) M^{\vec{s}}_{\zeta^{[n-1]}, \zeta\xi_n}(\omega)\\
    &=M^{\vec{s}}_{\zeta^{[n-1]}\zeta\xi_n, \zeta}(\omega) M^{\vec{s}}_{\zeta^{[n-1]}\zeta, \xi_n}(\omega) M^{\vec{s}}_{\zeta^{[n-1]}, \zeta}(\omega) \,.
\end{split}
\end{equation}
Since by hypothesis $\zeta$ has a non-empty set of good return words, by definition~\ref{def:good_ret_word}  for all $\beta \in \mathcal{A}$ we have that 
\begin{equation}
    \zeta(\beta)= p^{\beta} vc q^\beta
\end{equation}
for some $v\in GR(\zeta)$ and some words $p^\beta, q^\beta$. Therefore for any substitution $\xi$ we have
\begin{equation}
    \left|\Phi_{c}^{S^{t}_\xi \vec{s}}(\zeta(\beta), \omega)\right| \leq S^{t}_{\zeta}(\beta, c)-2+\left|1+\exp(-2\pi\imath \omega |v|_{S^{t}_\xi\vec{s}})\right| \leq S^{t}_{\zeta}(\beta, c)-\frac{1}{2}\|\omega |\xi(v)|_{\vec{s}}\|_{\mathbb{R}/\mathbb{Z}}^{2}.
\end{equation}
since the left hand side is the some of $S^{t}_{\zeta}(\beta, c)$ many trigonometric terms and $v=p^{\beta} vc q^{\beta}$.
Thus for any substitution $\xi$ we have
\begin{equation}
    \left|M_{\xi, \zeta}^{\vec{s}}(\omega)(\beta, c)\right| \leq S^{t}_{\zeta}(\beta, c)-\frac{1}{2}\|\omega |\xi(v)|_{\vec{s}}\|_{\mathbb{R}/\mathbb{Z}}^{2}.
\end{equation}
For $\vec{x}>\vec{0}$ we have the following estimate
\begin{equation}
\begin{split}
    \left(\left| M^{\vec{s}}_{\xi, \zeta}(\omega)\right| \cdot \vec{x}\right)_{\alpha}&= \sum_{\beta \in \mathcal{A}} \left|M_{\xi, \zeta}(\omega)(\alpha, \beta)\right| \cdot x_{\beta}\\
    & \leq \sum_{\beta \in \mathcal{A}}S^{t}_{\zeta}(\alpha, \beta)\cdot x_{\beta}-\frac{1}{2} \left\|\omega |\xi(v)|_{\vec{s}}\right\|^{2}_{\mathbb{R}/\mathbb{Z}} \cdot x_c\\
    & \leq \left(1-c_2 \cdot col(\vec{x})\|\omega|\xi(v)|_{\vec{s}}\|_{\mathbb{R}/\mathbb{Z}}^{2}\right) \sum_{\beta \in \mathcal{A}}S^{t}_{\zeta}(\alpha, \beta) x_\beta\\
    &=\left(1-c_2 \cdot col(\vec{x})\|\omega|\xi(v)|_{\vec{s}}\|_{\mathbb{R}/\mathbb{Z}}^{2}\right) \left(S^{t}_{\zeta}\vec{x}\right)_\alpha
\end{split}
\end{equation}
where
\begin{equation}
    c_2=\frac{1}{2d \max_{\gamma, \beta} Q_{\gamma, \beta}}, \;\; \quad col(\vec{x})=\frac{\min_{\beta} x_\beta}{\max_\beta x_\beta}. 
\end{equation}
Thus,
\begin{equation}
    \left| M^{\vec{s}}_{\xi, \zeta}(\omega)\right| \cdot \vec{x} \leq \left(1-c_2 \cdot col(\vec{x})\|\omega|\xi(v)|_{\vec{s}}\|_{\mathbb{R}/\mathbb{Z}}^{2}\right) S^{t}_{\zeta}\vec{x}
\end{equation}
for any arbitrary return word $v\in GR(\zeta)$. Therefore applying the above to $\vec{x}=\left(S^{[n-1]}\right)^{t}{\vec{1}}\in Q^{t}\mathbb{R}_{+}^{d}$ (here we have used that $\zeta_n$ ends with $\zeta$) and $\xi=\zeta^{[n-1]}$
\begin{equation}
\begin{split}
    \left| M_n\right| \left(S^{[n-1]}\right)^{t}\vec{1} &\leq S^{t}_{\zeta} S^{t}_{\xi_n} \left(1-c_1. \max_{v\in GR(\zeta)} \|\omega|\xi(v)|_{\vec{s}}\|_{\mathbb{R}/\mathbb{Z}}^{2}\right)S^{t}_{\zeta}\left(S^{[n-1]}\right)^{t} \vec{1}\\
    &=\left(1-c_1 \cdot \max_{v\in GR(\zeta)} \|\omega|\xi(v)|_{\vec{s}}\|_{\mathbb{R}/\mathbb{Z}}^{2}\right) \left(S^{[n]}\right)^t\vec{1}
\end{split}
\end{equation}
Iterating this inequality yields the desired result.
\end{proof}
It is easy to see that for a function $f_{\alpha}^{(\ell)}(x,t)=  \mathbbm{1_{\mathfrak{X}^{(\ell)}_\alpha}}(x,t)\cdot \psi_{\alpha}^{(\ell)}(t)$ we have
\begin{equation}
    \label{expbound}
    \left|S_{R}^{(x,0)}(f,\omega)\right|= \left|\widehat{\psi_\alpha^{(\ell)}}(\omega)\right|\cdot \left|\Phi_{\alpha}^{\vec{s}^{(\ell)}}(x^{(\ell)}[0,N-1], \omega)\right| \leq \left\|f\right\|_{\infty} \cdot
    \left|\Phi_{\alpha}^{\vec{s}^{(\ell)}}(x^{(\ell)}[0,N-1], \omega)\right|, 
\end{equation}
where $R=\left|x^{(\ell)}[0,N-1]\right|_{\vec{s}}$.
Hence it is plausible that bounds for exponential sums could be converted to bounds for 
twisted Birkhoff integrals (sums) which in turn give local asymptotics for the associated spectral measures.

 To pass from bounds for exponential sums corresponding to substitution words $\zeta^{[\ell+1,k]}(\beta)$ to the ones corresponding to general words we need the following prefix-suffix lemma. Throughout, we will denote by $x^{(\ell)}$ the elements of the substitution space $X_{\sigma^{\ell}\mathrm{a}}$ and note that ($\mathfrak{X}^{\vec{s}}_{\mathrm{a}}, h_{t}), (\mathfrak{X}^{\vec{s}^{(\ell)}}_{\sigma^{\ell}\mathrm{a}}, h_{t})$ are isomorphic systems.

 \begin{lemma} (Lemma 3.6, \cite{MR3773061})  For any $x^{(\ell)}\in X_{\sigma^{\ell}\mathrm{a}}$ and $N \geq 1$ we have
 \begin{equation}
 \label{tower}
     x^{(\ell)}[0,N-1]=\zeta^{[\ell+1]}(u_{\ell+1})\zeta^{[\ell+1, \ell+2]}(u_{\ell+2})\dots\zeta^{[\ell+1,n]}(u_{n})\zeta^{[\ell+1,n]}(v_{n})\dots\zeta^{[\ell+1]}(v_{\ell+1})\,,
 \end{equation}
 where $x^{(\ell)}[0,N-1]$ denotes the word consisting in the first $N$ symbols in the symbolic representation of $x^{(\ell)}$ and $u_j, v_j$, $j=\ell+1,\dots,n$, are respectively proper suffixes and prefixes of words of the form $\zeta_{j+1}(\beta), \beta \in \mathcal{A}$. The words $u_j, v_j$ may be empty, except that at least one of $u_n,v_n$ is nonempty. Moreover, 
 \begin{equation}
     \label{suff}
     \underset{\beta \in \mathcal{A}}{\min}\, \left|\zeta^{[\ell+1,n]}(\beta)\right| \leq N \leq 2\, \underset{\beta \in \mathcal{A}}{\max} \,  \left|\zeta^{[\ell+1,n+1]}(\beta)\right|.
 \end{equation}
 \end{lemma}
 
 Note that as $u_{j}, v_{j}$ are prefixes and suffixes of the words $\zeta_{j+1}(\beta)$ for some $\beta \in \mathcal{A}$ their lengths do not exceed $\|S_{j+1}\|_{1}$ and therefore the previous lemma and proposition yield the following
 
 \begin{prop}
 \label{goodbound}
  Under the assumptions of Proposition~\ref{mainbound}, for any $\ell \geq 1$, $\alpha \in \mathcal{A}, N \in \mathbb{N}, \vec{s}> \vec{0},$ and $\omega \in \mathbb{R}$, we have,
   \begin{equation}
    \begin{split}
    \left|\Phi_{\alpha}^{\vec{s}^{(\ell)}}(x^{(\ell)}[0,N-1], \omega)\right| &\leq 
    \|S^{[\ell+1,n]}\|_{1} (\left|u_n\right|+\left|v_n\right|) \\
    &\times \prod_{\ell +1 \leq k \leq n-1} \left(1-c_{1} \cdot \underset{v\in GR(\zeta)}{\max} \| \omega |\zeta^{[k]}(v)|_{\vec{s}}\|^{2}_{\mathbb{R}/\mathbb{Z}}\right)\\
    &+2 \sum_{j=\ell}^{n-1} \|S^{[\ell+1,j]}\|_{1} \cdot \|S_{j+1}\|_{1} \\
    &\times \prod_{\ell +1 \leq k \leq j-1} \left(1-c_{1} \cdot \underset{v\in GR(\zeta)}{\max} \| \omega |\zeta^{[k]}(v)|_{\vec{s}}\|^{2}_{\mathbb{R}/\mathbb{Z}}\right) \,,
   \end{split}
  \end{equation}
 where $u_n, v_n$ are as in \ref{tower}, $c_1$ is as in Proposition \ref{mainbound}, and $n\in \mathbb{N}$ is such that \eqref{suff} holds. 
 \end{prop}

\begin{prop}
\label{genbound}
Let $\vec{s}\in \mathbb{R}_{+}^{d}$. For every $\mathrm{a} \in \Omega'_{\mathbf{q}}$, and for all $\ell \geq \ell_0(\mathrm{a})$, and any bounded cylindrical function $f^{(\ell)}$ of level $\ell$, for any $(x,t) \in \mathfrak{X}^{\vec{s}}_{\mathrm{a}}$, with $x \in X_{\mathrm{a}}$, and $\omega \in \mathbb{R}$, for all $R\geq e^{4\theta_1 \ell}$, we have
 \begin{equation}
 \label{prb}
 \begin{split}
     \left|S_{R}^{(x,t)}(f^{(\ell)}, \omega)\right| &\leq C(\vec{s}, Q) \cdot \|f^{(\ell)}\|_{\infty} 
     \\ &\times \left(R^{1/2}+ R\prod_{\ell+1 \leq k \leq \frac{\log(R)}{2\theta_1}} \left(1-c_{1}.\underset{v \in GR(\zeta)}{\max}\|\omega |\zeta^{[k]}(v)|_{\vec{s}}\|^{2}_{\mathbb{R}/\mathbb{Z}} \right)\right)\,.
\end{split}
 \end{equation}
 
\end{prop}  
 
 \begin{proof}
As $S_{R}^{(y)}(\cdot,\omega)$ is linear on the first coordinate, we may assume w.l.o.g.  that 
$f^{(\ell)}=f^{(\ell)}_{\alpha}=\mathbbm{1_
{\mathfrak{X}^{(\ell)}_\alpha}}(x,t) \cdot \psi_{
\alpha}^{(\ell)}(t)$. The isomorphism between 
$\mathfrak{X}^{\vec{s}}_\mathrm{a}$ and  
$\mathfrak{X}^{\vec{s}^{(\ell)}}_\mathrm{\sigma^{\ell}\mathrm{a}}$ implies the existence of $x^{(\ell)} \in X_{\sigma^{\ell}\mathrm{a}}$ and $t' \in [0,s^{(\ell)}_{\max}]$ such that $h_{t'}(x^{(\ell)},0)=(x,t)$. Then
\begin{equation}
    \left|S_{R}^{(x,t)}(f^{(\ell)}_{\alpha}, \omega)\right|= \left|\int_{0}^{R} e^{-2\pi\imath \omega \tau} f^{(\ell)}_{\alpha} \circ h_{\tau+ t'}(x^{(\ell)},0)d\tau \right|= \left|\int_{t'}^{R+t'} e^{-2\pi\imath\omega \tau} f^{(\ell)}_{\alpha} \circ h_{\tau}(x^{(\ell)},0) d\tau \right|
\end{equation}
We let $s^{(\ell)}_{\max}$ and $s^{(\ell)}_{\min}$ be the maximal and minimal components of $\vec{s}^{(\ell)}=(S^{[\ell]})^{t}\vec{s}$ and remark that $t' \leq s^{(\ell)}_{\max}$. Therefore,
\begin{equation}
    \left|S_{R}^{(x,t)}(f^{(\ell)}_{\alpha}, \omega) - S_{R}^{(x',0)}(f^{(\ell)}_{\alpha}, \omega)\right| \leq 2 \|f^{(\ell)}_{\alpha}\|_{\infty}s^{(\ell)}_{\max}.
\end{equation}
 Now we let $N$ be the biggest natural number such that $R':=|x^{(\ell)}[0,N-1]|_{\vec{s}^{(\ell)}} \leq R$. Then $R-R' \leq s^{(\ell)}_{\max}$ and we get
 \begin{equation}
     \left|S_{R}^{(x^{(\ell)},0)}(f^{(\ell)}_{\alpha}, \omega) -S_{R'}^{(x^{(\ell)},0)}(f^{(\ell)}_{\alpha}, \omega)\right| \leq \|f^{(\ell)}_{\alpha}\|_{\infty}s^{(\ell)}_{\max}.
 \end{equation}
 Since $\mathrm{a} \in \Omega'_{\mathbf{q}}$ we have
 \begin{equation}
     \left|\frac{\log\|S^{[\ell]}\|_{1}}{\ell} - \theta_{1}\right| \leq \frac{\theta_{1} }{8}, \indent \forall \ell \geq \ell_{0}(\mathrm{a})\,.
 \end{equation}
 Then since $(S^{[\ell]})^{t}\vec{s}=\vec{s}^{(\ell)}$ we derive the following bounds:
\begin{equation}
    s^{(\ell)}_{\max}\leq s_{\max}\|S^{[\ell]}\|_{1} \leq s_{\max}e^{2\theta_{1}\ell} \leq s_{\max}R^{1/2}.
\end{equation}
where we denote by $s_{\min}$ and $s_{\max}$ the smallest and the largest entry of the vector $\vec{s}$. Thus, the first term in the right-hand side of \eqref{prb} bounds the cumulative errors due to changing $R$. Therefore the proof of Proposition~\ref{genbound} is concluded by the following proposition.
 \end{proof}
\begin{prop}
 Assume that the assumptions of Proposition~\ref{genbound} are satisfied. Then for all $\ell \geq \ell_{0}= \ell_{0}(\mathrm{a})$, and any bounded cylindrical function $f^{(\ell)}$ of level $\ell$, and any $x^{(\ell)} \in X_{\sigma^{\ell}\mathrm{a}}$
\begin{equation}
\label{goobbbb}
     \left|S_{R}^{(x^{(\ell)},0)}(f^{(\ell)}, \omega)\right| \leq C(\vec{s}, Q)\cdot \|f^{(\ell)}\|_{\infty} R \prod_{\ell+1 \leq k \leq 
     \frac{\log(R)}{2\theta_1}} 
     \Big(1-c_{1} \cdot\underset{v \in 
     GR(\zeta)}{\max}\|\omega |\zeta^{[k]}(v)|_{\vec{s}}\|^{2}_{\mathbb{R}/\mathbb{Z}} \Big)
\end{equation}
 whenever
 \begin{equation}
     R=|x^{(\ell)}[0,N-1]|_{\vec{s}^{(\ell)}} \geq e^{4\theta_1 \ell}.
 \end{equation}
\end{prop} 

\begin{proof}
 Again, without loss of generality, by linearity we may and do assume that the cylindrical function $f^{(\ell)}=f^{(\ell)}_{\alpha}$, where $f^{(\ell)}_{\alpha}(x,t)= \psi_{\alpha}^{(\ell)}(t) \cdot \mathbbm{1}_{\mathfrak{X}^{(\ell)}_{\alpha}}$, hence
 \begin{equation}
     S^{(x',0)}_{R}(f_{\alpha}^{(\ell)}, \omega)= \widehat{\psi}^{(\ell)}_{\alpha}(\omega)\cdot \Phi_{\alpha}^{\vec{s}^{(\ell)}}(x^{(\ell)}[0,N-1], \omega). 
 \end{equation}
Note that
\begin{equation}
    |\widehat{\psi}_{\alpha}^{(\ell)}(\omega)| \leq \|\psi_{\alpha}^{(\ell)}\|_1 \leq \|\psi_{\alpha}^{(\ell)}\|_{\infty} s_{\alpha}^{(\ell)} \leq \|f_{\alpha}^{(\ell)}\|_{\infty} s_{\max}^{(\ell)}\,,
\end{equation}
and by Proposition~\ref{goodbound} we have
\begin{equation}
\label{ingobound}
    \begin{split}
    \left|\Phi_{\alpha}^{\vec{s}^{(\ell)}}(x^{(\ell)}[0,N-1], \omega)\right| &\leq 
    \|S^{[\ell+1,n]}\|_{1} (\left|u_n\right|+\left|v_n\right|) \\
    &\times \prod_{\ell +1 \leq k \leq n-1} \left(1-c_{1} \cdot \underset{v\in GR(\zeta)}{\max} \| \omega |\zeta^{[k]}(v)|_{\vec{s}}\|^{2}_{\mathbb{R}/\mathbb{Z}}\right)\\
    &+2 \sum_{j=\ell}^{n-1} \|S^{[\ell+1,j]}\|_{1} \cdot \|S_{j+1}\|_{1} \\
    &\times \prod_{\ell +1 \leq k \leq j-1} \left(1-c_{1} \cdot \underset{v\in GR(\zeta)}{\max} \| \omega |\zeta^{[k]}(v)|_{\vec{s}}\|^{2}_{\mathbb{R}/\mathbb{Z}}\right) \,,
   \end{split}
\end{equation}
where 
\begin{equation}
    \label{Nb}
     \underset{\beta \in \mathcal{A}}{\min}\, |\zeta^{[\ell+1,n]}(\beta)| \leq N \leq 2\, \underset{\beta \in \mathcal{A}}{\max} \,  |\zeta^{[\ell+1,n+1]}(\beta)|.
\end{equation}
Since all the renormalization matrices begin and end with $Q=S_{\zeta}$ we obtain
\begin{equation}
    \label{lj}
    3\|S^{[\ell+1,j]}\|_{1} \leq\|S^{[\ell+1,j]}QS_{\xi_{j+1}}Q\|_{1} = \|S^{[\ell+1,j+1]}\|_{1}
\end{equation}
Notice that for two matrices $A,B$ with positive entries we have
\begin{equation}
\label{1b}
     col(A^{t})^{-1}\|A\|_{1}\|B\|_{1} \leq \|AB\|_{1} \leq  \|A\|_{1}\|B\|_{1}.
\end{equation}
Thus, \ref{1b} applied to $S^{[\ell+1, j]}$ and $S_{j+1}$ for $\ell \leq j \leq n-1$ gives
\begin{equation}
\label{expret}
    \|S^{[\ell+1, j]}\|_{1} \|S_{j+1}\|_{1} \leq col(Q^{t})\|S^{[\ell+1, j+1]}\|_{1} 
\end{equation}
Therefore, \ref{expret} and successive application of \ref{lj} yield the following
\begin{equation}
     \|S^{[\ell+1, j]}\|_{1} \|S_{j+1}\|_{1} \leq C_3\frac{\|S^{[\ell+1, n]}\|_{1}}{3^{n-j}}
\end{equation}
where $C_3=3 col(Q^{t}).$ Also note that the elements in the product in \ref{goobbbb} are at least $3/4$ and therefore
\begin{equation}
\begin{split}
    &\|S^{[\ell+1, j]}\|_{1} \|S_{j+1}\|_{1} \prod_{\ell +1 \leq k \leq j-1} \left(1-c_{1}.\underset{v\in GR(\zeta)}{\max} \| \omega |\zeta^{[k]}(v)|_{\vec{s}}\|^{2}_{\mathbb{R}/\mathbb{Z}}\right)\\
    & \leq C_3 (\frac{4}{9})^{(n-j)} \|S^{[\ell+1,n]}\|_{1} \prod_{\ell +1 \leq k \leq n-1} \left(1-c_{1}.\underset{v\in GR(\zeta)}{\max} \| \omega |\zeta^{[k]}(v)|_{\vec{s}}\|^{2}_{\mathbb{R}/\mathbb{Z}}\right),
\end{split}
\end{equation}
and so
\begin{equation}
\begin{split}
2 \sum_{j=\ell}^{n-1} \|S^{[\ell+1,j]}\|_{1} \cdot \|S_{j+1}\|_{1} & \prod_{\ell +1 \leq k \leq j-1} \left(1-c_{1}.\underset{v\in GR(\zeta)}{\max} \| \omega |\zeta^{[k]}(v)|_{\vec{s}}\|^{2}_{\mathbb{R}/\mathbb{Z}}\right) \\
& \leq C_4\|S^{[\ell+1,n]}\|_{1} \prod_{\ell +1 \leq k \leq n-1} \left(1-c_{1}.\underset{v\in GR(\zeta)}{\max} \| \omega |\zeta^{[k]}(v)|_{\vec{s}}\|^{2}_{\mathbb{R}/\mathbb{Z}}\right),
\end{split}
\end{equation}
where $C_4=4C_3 \sum_{i=0}^{\infty}\left(\frac{4}{9}\right)^{i}$ only depends on the word $\mathbf{q}$. So the top term in \ref{ingobound} dominates the sum up to a constant.

In order to conclude the proof we need to find appropriate bounds for $\|S^{[\ell+1,n]}\|_{1}$ and $n$ in terms of $R$.
As $R=|x^{(\ell)}[0,N-1]|_{s^{(\ell)}}$ it can be readily seen that
\begin{equation}
\label{Rbs}
  Ns^{(\ell)}_{\min} \leq R \leq Ns^{(\ell)}_{\max}.
\end{equation}
Again, from the fact that all renormalization matrices end with $Q=S_\zeta$ we derive that
\begin{equation}
    s^{(\ell)}_{\max} \leq col(Q^{t}) \cdot s^{(\ell)}_{\min}\,,
\end{equation}
hence~\eqref{Nb} implies that 
\begin{equation}
   col(Q^{t})^{-1} \cdot\|S^{[\ell+1,n]}\|_{1} \leq N \,,
\end{equation}
and therefore
\begin{equation}
\label{Rb}
col(Q^{t})^{-2} \cdot \|S^{[\ell+1,n]}\|_{1} \cdot s^{(\ell)}_{\max}\leq  N s^{(\ell)}_{\min} \leq R,
\end{equation}

\begin{equation}
    \min_{\beta \in \mathcal{A}} \left| \zeta^{[\ell+1,n]}(\beta)\right| \left(\left|u_n\right|+\left|v_n\right|\right) \leq N.
\end{equation}
As $S^{[\ell+1, n]}$ is both column-balanced and row-balanced i.e.,
\begin{equation}
    \max\left\{ col(S^{[\ell+1,n]}), col((S^{[\ell+1,n]})^{t}) \right\} \leq \max\left\{ col(Q), col(Q^{t}) \right\},
\end{equation}
there exists a constant $C_5>0$ only depending on $\mathbf{q}$ such that
\begin{equation}
    \|S^{[\ell+1, n]}\|_1 \leq C_5  \min_{\beta \in \mathcal{A}} \left| \zeta^{[\ell+1,n]}(\beta)\right|.
\end{equation}
Thus,
\begin{equation}
     \|S^{[\ell+1, n]}\|_1 s^{(\ell)}_{\max}\left(\left|u_n\right|+\left|v_n\right|\right) \leq C_5N s^{(\ell)}_{\max} \leq C_6 R
\end{equation}
for some constant $C_6>0$.  
On the other hand, \eqref{Nb}, \eqref{Rbs} and \eqref{1b}   yield
\begin{equation}
   R \leq Ns^{(\ell)}_{\max} \leq 2\|S^{[\ell+1,n+1]}\|_{1}\|S^{[\ell]}\|_{1}s_{\max} \leq 2s_{\max}col(Q^{t})\|S^{[n+1]}\|_{1}.
\end{equation}
By assumption $\mathrm{a} \in \Omega'_{\mathbf{q}}$, we get
\begin{equation}
    R \leq 2s_{\max}col(Q^{t})e^{(n+1)\theta_1(1+1/2)}\leq e^{2n\theta_1}.
\end{equation}
and therefore $\log(R)/2\theta_1 \leq n$ and the proof concludes.
\end{proof}
In the following lemma, we establish the transition from cylindrical functions to general weakly Lipschitz functions.

\begin{lemma}
\label{limbound}
Let $\vec{s} \in \mathbb{R}_{+}^{d}$, $\gamma \in (0,1)$, $f\in 
Lip_{w}(\mathfrak{X}^{\vec{s}}_{\mathrm{a}})$, and $\mathrm{a}\in \Omega'_{\mathbf{q}}$. Let $B<C$ be two positive 
reals and assume that, for all $\ell \geq \ell(\mathrm{a}, \vec{s}, 
B,C)$, and all bounded cylindrical functions $f^{(\ell)}$ of level $\ell$ we have
\begin{equation}
\label{lbound}
    \left|S_{R}^{(x,t)}(f^{(\ell)}, \omega)\right| \leq C(\mathrm{a}) \|f^{(\ell)}\|_{\infty} \cdot  R^{1-\gamma}\,,
\end{equation}
for all $\omega \in [B,C]$ and $R \geq e^{\gamma^{-1}\theta_1 \ell}$. Then, for any weakly Lipschitz function $f$ on $\mathfrak{X}_{\mathrm{a}}^{\vec{s}}$, we have
\begin{equation}
    \left|S^{(x,t)}_{R}(f, \omega)\right| \leq \Tilde{C}(\mathrm{a})\|f\|_{L} \cdot  R^{1-\gamma}\,,
\end{equation}
for all $\omega \in [B,C]$ and $R \geq R(\mathrm{a}, \vec{s}, B, C):=e^{\gamma^{-1}\theta_1 \ell_2},$ where $\ell_2:=\max\{\ell_{0}(\mathrm{a}), \ell(\mathrm{a}, \vec{s}, B, C) \}$.

\end{lemma}
\begin{proof} Let $f^{(\ell)}$ be such that 
\begin{equation}
    f^{(\ell)}(x,t):= f(x_{\alpha},t) \,,
\end{equation}
whenever $x \in X_{\sigma^{\ell} \mathrm{a}}$ belongs to $\zeta^{[\ell]}[\alpha]$ and $x_\alpha \in \zeta^{[\ell]}[\alpha]$ is arbitrarily chosen. Then
\begin{equation}
    \|f(x,t)-f^{(\ell)}(x,t)\|_{\infty} \leq \|f\|_{L} \cdot \underset{\alpha \in \mathcal{A}}{\max}\, \mu_{\mathrm{a}}(\zeta^{[\ell]}[\alpha]).
\end{equation}
and $\|f^{(\ell)}\|_{\infty} \leq \|f\|_{\infty}$. Since $Area(\mathrm{a}, \vec{s}):=\sum_{\alpha \in \mathcal{A}} \mu([\alpha])\cdot s_{\alpha}$ remains fixed under induction, we have
\begin{equation}
    \max_{\alpha \in \mathcal{A}} \mu_{\mathrm{a}}(\zeta^{[\ell]}[\alpha]) \cdot  col(Q^{t})^{-1} s_{\min} \|S^{[\ell]}\|_{1} \leq \max_{\alpha \in \mathcal{A}} \mu_{\mathrm{a}}(\zeta^{[\ell]}[\alpha]) \cdot s^{(\ell)}_{\min} \leq Area(\mathrm{a}, \vec{s})\,.
\end{equation}
Since $\mathrm{a} \in \Omega'_{\mathbf{q}}$, we then obtain, for all $\ell \geq \ell_{0}(\mathrm{a})$, the inequality
\begin{equation}
    \max_{\alpha \in \mathcal{A}} \mu_{\mathrm{a}}(\zeta^{[\ell]}[\alpha]) \leq s_{\min}^{-1}s_{\max} col(Q^{t}) e^{-\ell \theta_1 (1-1/8)}\,.
\end{equation}
Thus, for all $\ell \geq \ell_{0}(\mathrm{a})$, we have
\begin{equation}
\label{approx}
    \|f-f^{(\ell)}\|_{\infty} \leq \|f\|_{L} e^{-3\ell \theta_1/4}\,,
\end{equation}
hence, if we let $\ell_{1}(\mathrm{a}, \vec{s}, B, C):=\max\{\ell_{0}(\mathrm{a}), \ell(\mathrm{a}, \vec{s}, B, C)\}$ and $R(\mathrm{a}, \vec{s}, B, C):= e^{\gamma^{-1} \theta_1 \ell_1}$, and, for $R \geq R(\mathrm{a}, \vec{s}, B, C)$, we let 
\begin{equation}
     \ell:= \Big\lfloor{\frac{2\gamma \log R}{\theta_1}}\Big\rfloor \,,
\end{equation}
it is then immediate to see that $\ell \geq \ell_1$, and therefore, by \eqref{approx}, 
\begin{equation}
\label{funcapprox}
    \left|S_{R}^{(x,t)}(f, \omega)-S_{R}^{(x,t)}(f^{(\ell)}, \omega)\right| \leq R \cdot \|f\|_{L}e^{-3\ell \theta_1/4} \leq \|f\|_{L} R^{1-\gamma}\,,
\end{equation}
which together with \eqref{lbound} gives the desired result.
\end{proof}
\section{Quantitative Veech criterion}
\label{sec:eff_Veech}
Note that, as we assumed that $\vec{\ell}(v)$ for $v \in GR(\zeta)$ generate $\mathbb{Z}^{d}$, it can be easily seen that the pseudo-norms 
obtained by $\underset{v\in GR(\zeta)}{\max} \big\langle \vec{\ell}(v), \vec{x} \big\rangle$, where $\vec{x} \in \mathbb{R}^{d}$ is comparable to the pseudo-norm $\|\vec{x}\|_{\mathbb{R}^{d}/\mathbb{Z}^{d}}$. In fact, we have the following simple lemma.
\begin{lemma} Let $\{v_{j}\}_{j=1}^{k}$ be good return words for $\zeta$, such that $\{\vec{\ell}(v_{j})\}_{j=1}^{k}$ generate $\mathbb{Z}^{d}$. Then there exists a constant $C_{\zeta}>1$ such that
\begin{equation} 
    \label{norm}
    C_{\zeta}^{-1}\|\vec{x}\|_{\mathbb{R}^{d}/\mathbb{Z}^{d}} \leq \underset{1\leq j \leq k}{\max} \big\langle \vec{\ell}(v_{j}), \vec{x} \big\rangle \leq C_{\zeta} \|\vec{x}\|_{\mathbb{R}^{d}/\mathbb{Z}^{d}}.
\end{equation}
\end{lemma}
The following equality 
\begin{equation}
    \|\omega |\zeta^{[k]}(v)|_{\vec{s}}\|_{\mathbb{R}/\mathbb{Z}} =\| \omega  \big\langle \vec{\ell}(v), (S^{[k]})^{t} \big\rangle \|_{\mathbb{R}/\mathbb{Z}}= \|\big\langle \vec{\ell}(v), \mathbb{A}(k,\mathrm{a}) \omega\vec{s}\big\rangle\|_{\mathbb{R}/\mathbb{Z}}
\end{equation}
together with \eqref{norm} imply that
\begin{equation}
    \prod_{\ell+1 \leq k \leq \frac{\log R}{2\theta_1}}  \left(1-c_{1}\cdot \underset{v\in GR(\zeta)}{\max}\|\omega| \zeta^{[k]}(v)|_{\vec{s}}\|^{2}_{\mathbb{R}/\mathbb{Z}}\right) \leq \prod_{\ell+1 \leq k \leq \frac{\log R}{2\theta_1}}  \left(1-c'_{1}\cdot \|\mathbb{A}(k,\mathrm{a})(\omega\vec{s})\|^{2}_{\mathbb{R}^{d}/\mathbb{Z}^{d}}\right),
\end{equation}
for some constant $c'_{1}$ depending on $\zeta$. 

We now state a quantitative version of Veech criterion, already apparent in the work of Forni \cite{2019arXiv190811040F}, and explicitly introduced by Bufetov and Solomyak in \cite{2019arXiv190809347B}. Our version may be seen as a slight generalization of that of
\cite{2019arXiv190809347B}, in the sense that we do not require Oseledets regularity.
\begin{thm}
\label{QVC}
(Quantitative Veech Criterion) Let $\vec{s}\in \mathbb{R}_{+}^{d}$ and $\mathrm{a}\in \Omega'_{\mathbf{q}}$. Let $B<C$ be two fixed positive real numbers. Assume that there exist  $N(\mathrm{a}, \vec{s}, B,C) \in \mathbb N$  and $\epsilon>0$ such that
\begin{equation}
    \#\left\{1 \leq i \leq N: \|\mathbb{A}(i,\mathrm{a})(\omega \vec{s})\|_{\mathbb{R}^{d}/\mathbb{Z}^{d}}> \epsilon\right\} \geq \epsilon N, \indent \forall N \geq N(\mathrm{a}, \vec{s}, B,C)\,,
\end{equation}
uniformly for all $\omega \in [B,C]$.  Let 
\begin{equation}
    \gamma= \min\left\{\frac{\epsilon}{16}, \frac{-\epsilon \log(1-c'_{1}\epsilon^{2})}{8\theta_1}\right\} 
\end{equation}
and assume that $\mathrm{a}\in \Omega'_{\mathbf{q}}$.
Then there exists $R(\mathrm{a}, \vec{s}, B, C)>0$ such that, for every weakly Lipschitz function $f$ and every $(x,t) \in \mathfrak{X}^{\vec{s}}_{\mathrm{a}}$, we have
\begin{equation}
    \label{gammab}
    \left|S^{(x,t)}_{R}(f, \omega)\right| \leq \Tilde{C}(\mathrm{a})\|f\|_{L} \cdot  R^{1-\gamma}, \indent \forall \omega \in [B,C],\; \forall R \geq R(\mathrm{a}, \vec{s}, B, C)\,,
\end{equation}
which in turn yields the following H\"older bound
\begin{equation}
\sigma_{f}([\omega-r, \omega+r]) \leq C(\mathrm{a}) \|f\|^{2}_{L} \cdot r^{2\gamma}, \indent \forall \omega \in [B,C],\; 0<r\leq (2R(\mathrm{a}, \vec{s}, B, C))^{-1}\,.
\end{equation}
Moreover, $R(\mathrm{a}, \vec{s}, B, C)$ is given by $e^{\gamma^{-1} \ell_1 \theta_1}$ where $\ell_1:=
\max \{ \ell_{0}(\mathrm{a}), \lceil{2\gamma N(\mathrm{a},\vec{s}, B, C)\rceil+1)}\}$.
\end{thm}

\begin{proof} We proceed by showing that the conditions of 
Lemma~\ref{limbound} 
hold for a suitable choice of $\ell(\mathrm{a}, \vec{s}, B, C)$ given in terms of $N_{0}:=N(\mathrm{a}, \vec{s}, B, C)$. By Proposition~\ref{genbound}, for all $\ell \geq \ell_{0}(\mathrm{a})$ and all $R \geq e^{4\theta_1 \ell}$,
we have
\begin{equation} 
\label{bb}
\left|S_{R}^{(x,t)}(f^{(\ell)}, \omega)\right| \leq C(\vec{s}, Q)\cdot \|f^{(\ell)}\|_{\infty}\left( R^{1/2}+R \prod_{\ell+1 \leq k \leq \frac{\log(R)}{2\theta_1}} \left(1-c'_{1}\cdot\|\mathbb{A}(k,\mathrm{a})(\omega \vec{s})\|^{2}_{\mathbb{R}^{d}/\mathbb{Z}^{d}} \right)\right).
\end{equation}
 We let $\ell_{1}= \max\{\ell_{0}(\mathrm{a}),\lceil 2\gamma(N_0+1) \rceil \} $. Then, for all $\ell \geq \ell_{1}$ and for all $R \geq e^{\gamma^{-1}\theta_{1} \ell} \geq  e^{4\theta_1 \ell}$, we may take $N=\big\lfloor{\frac{\log R}{2\theta_1}} \big\rfloor \geq N_{0}$. Hence $\ell \leq \gamma \log(R)/ \theta_1 \leq \epsilon \log(R) /16\theta_1 \leq \epsilon N/4$.  We therefore obtain
\begin{equation}
\begin{split}
    \prod_{\ell+1 \leq k \leq \frac{\log(R)}{2\theta_1}} \left(1-c'_{1}\cdot\| \mathbb{A}(k,\mathrm{a})(\omega \vec{s})\|_{\mathbb{R}^{d}/\mathbb{Z}^{d}}^{2} \right) &\leq (1-c'_{1}\epsilon^{2})^{\epsilon N-\ell-1} \\  
    &\leq (1-c'_{1}\epsilon^{2})^{\epsilon N/2} 
    \leq e^{-4\gamma \theta_1 N}
    \leq R^{-\gamma}.
\end{split}
\end{equation}
Together with $\eqref{bb}$, and Lemma~\ref{limbound}, we derive \eqref{gammab}, which combined with Lemma~\ref{Twisted Birkhoff Spectral}, concludes the proof of the H\"older bound. 
\end{proof}

\section{Derivation of the main results}
\label{sec:main_results}
\subsection{Upper bound for non-rotation type IETs}
\label{upper bound for non-rotation type}
 Note that for $0<t<1$ the specific form of $R(\mathrm{a}, \vec{s}, t, 1-t)$, given in \eqref{QVC}, and the bound for $N(\mathrm{a}, \vec{s}, t, 1-t)$, given in \eqref{IETsp}, yield that we may replace the bound in \eqref{gammab} with the following bound:
\begin{equation}
    \left|S_{R}^{(x,t)}(f, \omega)\right| \leq \tilde{C}(\mathrm{a}) \|f\|_{L} \cdot t^{-2C\gamma\theta_1} R^{1-\gamma}\,,
\end{equation}
for all $\omega \in [t, 1-t]$ and all $R\geq 1$,  and for some constant $C>0$ not depending on the IET.

From the above bound and from Corollary~\ref{cor LD}, 
Theorem~\ref{IETsp}, Definition~\ref{defi}, and Theorem~\ref{QVC}, we can easily derive the following theorem (we take $\vec{s}$ to be equal to $h=(1,1,\dots,1)$):

\begin{thm} 
\label{twisbirk}
For every $d>3$ and any irreducible permutation $\pi$ of $\{1, \dots,d\}$, not of rotation class, there exist $\gamma>0$, $\beta>0$ and a measurable set $\Delta_{spec} \subset \Delta=\mathbb{P}_{+}^{d-1}$, whose complement has positive Hausdorff codimension, with the property that, for all $\lambda \in \Delta_{spec}$, there exists a constant $C_\lambda>0$ such that,  for every Lipschitz continuous function $f$ on the interval $I=[0,1)$, for every $t\in (0,1)$ and for all $\theta\in [t, 1-t]$,  we have 
\begin{equation}
    \left|\sum_{n=0}^{N-1} e^{2\pi\imath n\theta}f(T_{\lambda, \pi}^{n}(x))\right| \leq C_{\lambda}\|f\|_{L}\cdot t^{-\beta} N^{1-\gamma}, \:\; \; \quad \forall N \geq 1, 
\end{equation}
for all $x\in I$. The above bound holds uniformly for all $\theta \in [t,1-t]$, which implies that 
\begin{equation}
    \sigma_f([\theta-r, \theta+r]) \leq C'_{\lambda}\|f\|^{2}_{L} t^{-2\beta}\cdot r^{\gamma} \;\; \: \quad \forall r\leq 1/2.
\end{equation}
\end{thm}
\begin{rmk}
Note that by the results of section \ref{sec:twist_int}, we get the above theorem for all points except the countably many points in the backward orbits of singularities, however, for Lipschitz observables the upper bound extends to the singularities simply by uniform continuity of the observable and the right continuity of the IET. 
\end{rmk}
\smallskip
\noindent This is in fact the content of Theorem~\ref{mm}, which immediately implies Theorem~\ref{specdim}.

\noindent
Let $(\lambda, \pi)$ be such that the above theorem holds and assume that we have
\begin{equation}
    \left\|\sum_{i=0}^{N-1} f(T_{\lambda, \pi}^{n}(x))\right\|_{2} \leq C''_{\lambda}(f) N^{1-\alpha},
\end{equation}
where $\alpha>0$ only depends on the Rauzy class of the permutation $\pi$ and not on the individual IET. 
Letting $U$ be the Koopman operator associated with $T_{\lambda, \pi}$, we have
\begin{equation}
\begin{split}
\label{+,-}
    I:=\sum_{n=0}^{N-1} \left|\left\langle U^{n}(f), g \right\rangle\right|^{2}
    &= \sum_{n=0}^{N-1}{\big\langle U^{n}(f), g\big\rangle} \int_{\mathbb{R}/\mathbb{Z}} e^{2\pi\imath n \theta} d\sigma_{f,g}(\theta)
    \\&=\int_{\mathbb{R}/\mathbb{Z}}
 \big\langle \sum_{n=0}^{N-1}e^{2\pi\imath  n \theta} U^{n}(f), g \big\rangle d\sigma_{f,g}(\theta)
 =I_{\epsilon}^{+}+I_{\epsilon}^{-} \,,
\end{split}
\end{equation}
where by the Cauchy-Schwartz inequality and the bound $|\sigma_{f,g}(0,1)| \leq \|f\|_2 \|g\|_2$, we have
\begin{equation}
\label{int+}
    I_{\epsilon}^{+}:=\int_{\epsilon}^{1-\epsilon} 
    \big\langle \sum_{n=0}^{N-1}e^{2\pi\imath n\theta} U^{n}(f), g \big\rangle d\sigma_{f,g}(\theta) \leq C_{\lambda}\|f\|_{L}\epsilon ^{-\beta}N^{1-\gamma}\|g\|_{2}^{2} \|f\|_{2} \,,
\end{equation}
 and also
\begin{equation}
\label{int-}
    I_{\epsilon}^{-}:=\int_{-\epsilon}^{\epsilon} 
    \big\langle \sum_{n=0}^{N-1}e^{2\pi\imath n\theta} U^{n}(f), g \big\rangle d\sigma_{f,g}(\theta) \leq \sigma_{f,g}((-\epsilon, \epsilon)) N \|g\|_{2}\|f\|_2\,.
\end{equation}
 If we take $\epsilon= N^{-\eta}$ for $\eta:= \gamma/(\alpha+ \beta)$, then by Lemma~\ref{Twisted Birkhoff Spectral} we get
\begin{equation}
    \left|\sigma_{f,g}(-\epsilon, \epsilon)\right| \leq 10C''_{\lambda}(f) \|g\|_{2} N^{-\alpha \eta} \,.
\end{equation}
Therefore
\begin{equation}
    I_{\epsilon}^{+} \leq C_{\lambda}\|f\|_{L} N^{1-\gamma+\beta \eta}\|g\|_{2}^{2}\|f\|_2 \,
    \quad \text{ and } \quad  
    I_{\epsilon}^{-} \leq 10C''_{\lambda}(f)N^{1-\alpha \eta} \|g\|_{2}^{2}\|f\|_{2}.
\end{equation}
By summing the two inequalities above, we get that
\begin{equation}
    I \leq K_{\lambda}(f) N^{1-\alpha'} \|g\|_{2}^2 \|f\|_{2} \,,
\end{equation}
where $\alpha':=\alpha\gamma/(\alpha+\beta)>0$. Also note that $K_\lambda(f)$ may be taken to be some constant, only depending on the IET, times the Lipschitz norm of $f$ as we shall see below.

\smallskip
Athreya and Forni~\cite{MR2437681} showed that there exists a fixed $\alpha>0$, depending only on the stratum to which a translation surface $S$ belongs, such that for almost every direction $\theta\in \mathbb T$ there exists a constant $K(\theta)>0$ depending measurably on $\theta$ such that the following holds. For any zero-average function $f \in H^{1}(S)$ (the 
space of Sobolev regular functions on the surface $S$ with square integrable first derivatives) the translation flow $\phi_{t}^{\theta}$ on $S$ in direction $\theta$ satisfies
\begin{equation}
\label{ergbound}
    \left|\int_{0}^{T} f\circ \phi^{\theta}_{t}(x) dt \right| \leq K(\theta) \|f\|_{H^{1}(S)}\, T^{1-\alpha}\,,
\end{equation}
for all points $x\in S$ that do not belong to the backward orbits of the singularities and for all $T>1$. 

We remark in passing that they obtained this result by showing that there is a gap between the top Lyapunov exponent of the Kontsevich-Zorich cocycle over the Teichm\"uller flow and the upper second Lyapunov exponent, for almost every direction as a consequence of a first variation formula for the Hodge norm and the fact that, for almost every direction $\theta$, the Teichm\"uller orbit $g_{t}r_{\theta}S$ spends a positive proportion of time inside certain compact sets ($g_t$ denotes the Teichm\"uller geodesic flow and $r_{\theta}$ the rotation of angle $\theta$).

This result on positive density recurrence was more recently generalized in \cite{2017arXiv171110542A} from a set of directions of full measure to a set of  positive Hausdorff codimension. Thus, for every flat surface $S$, the bound \eqref{ergbound} holds for a set of directions whose complement has Hausdorff dimension less than one.

\begin{figure}
\centering
\includegraphics{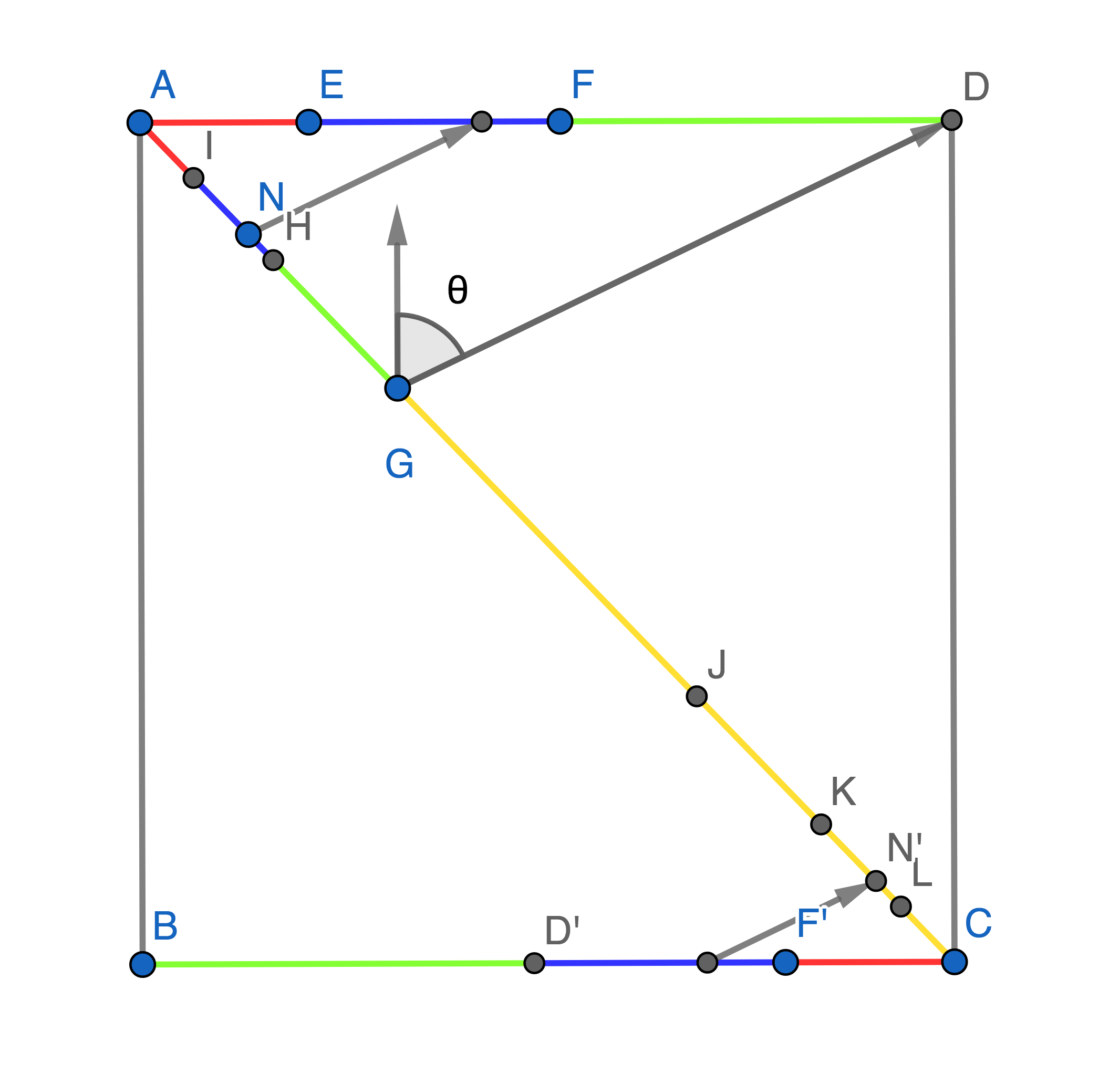}
\caption{The first return map to the diagonal of the translation flow in direction $\theta$.}
\label{figureIET}
\end{figure}

Let $f$ be a Lipschitz function with zero average. We show that $f$ is co-homologous to a Lipschitz function that is zero near the endpoints of the interval $[0,1]$ through a Lipschitz transfer function whose sup norm is bounded by $\|f\|_{\infty}$. Simply take some $\delta>0$ small enough such that the intervals $[0,\delta], [1-\delta, 1]$ and their images under $T_{\lambda, \pi}$ are all disjoint and at least a distance $\delta$ apart from each other. We may further assume that $T_{\lambda, \pi}$ is continuous on these intervals. We now define $f_1$ first on the union of these four intervals and extend it to the complement by the unique linear maps that make $f_1$ a continuous Lipschitz function on the whole interval. 
\begin{equation}
f_1(x)=
\begin{cases}
0 &\text{if} \; x \in [0,\delta]\cup [\delta, 1-\delta], \\
f(T^{-1}(x)) & \text{if}\; x\in T_{\lambda, \pi}\left([0,\delta]\cup [\delta, 1-\delta]\right).
\end{cases}
\end{equation}

Then the assumption that the intervals are at least $\delta$ apart from one another and the way $f_1$ is defined imply that
\begin{equation}
    \|f_1\|_{L} \leq  \frac{2}{\delta}\|f\|_{\infty}+\|f\|_{L}.
\end{equation}
In addition,
\begin{equation}
    f_2(x):=f(x)+f_1(x)-f_1(T_{\lambda, \pi}(x))
\end{equation}
is a Lipschitz function on $[0,1]$ that vanishes on the union of the intervals $[0, \delta]$ and $[1-\delta, 1]$ with 
\begin{equation}
    \|f_{2}\|_{L} \leq C_0(\lambda) \|f\|_{L}.
\end{equation}
Thus the Birkhoff sums of $f$ may be approximated by those of $f_2$ up to an error of $2 \|f_1\|_{\infty} \leq C_0 \|f\|_{L}$.

Let now $T_{\lambda, \pi}$ be an interval exchange transformation that is realizable as the first return map of a translation flow in direction $\theta$ (see figure \ref{figureIET}).  The translation structure in direction $\theta$ can be recovered by gluing parallelograms whose bases lie in the interval $I(\lambda, \pi)$ corresponding to $T_{\lambda, \pi}$. Let $h_\alpha$ denote the length of the edge of the polygon parallel to the flow direction (it corresponds to the return time to the interval $I$ under the translation flow for points in $I_\alpha$). 
Let $\psi:[0, \underset{\alpha \in \mathcal{A}}{\max}{h_\alpha}] \to \mathbb{R}^{\geq 0}$ be a smooth bump function supported in a compact small subinterval such that
\begin{equation}
    \int \psi dt=1,
\end{equation}
and the set
\begin{equation}
    \left\{\phi^{\theta}_{t}(x), x\in I, t \in supp(\psi) \right\},
\end{equation}
is away from the singularities of the translation surface $S$. 

We now let $\tilde{f} \in H^{1}(S)$ be the function that attains the value $f_2(x) \psi(t)$ at the point $\phi^{\theta}_t(x)$ for $x\in I, 0 \leq t \leq h(x)$ where $h(x)$ is the non-horizontal side-length of the parallelogram to which $x$ belongs. It may be readily seen that the Birkhoff sums of $f_2$ are equal (up to a bounded factor not depending on $f_2$) to the Birkhoff integrals of $\tilde{f}$. Hence the considerations regarding the shift from $f$ to the function $f_2$ and \eqref{ergbound} imply that there exists a constant $C_{\lambda}''>0$ such that

\begin{equation}
\label{ergbound2}
    \left| \sum_{n=0}^{N-1} f(T_{\lambda, \pi}(x))\right| \leq C''_\lambda \|f\|_{L} N^{1-\alpha},
\end{equation}
for all points $x$ whose orbit does not meet the singularities of 
$T_{\lambda, \pi}$ and for all $N\geq 1$.

\smallskip
We now fix a translation surface $S$ and from the polynomial ergodic property \eqref{ergbound} for  translations flows on $S$ in a set of directions with positive Hausdorff codimension, we derive that the corresponding property for IETs \eqref{ergbound2} holds for a subset of a given line in the space of IETs, whose complement has positive Hausdorff codimension. By decomposing the space of all IETs as a union of all these lines,  we get that the estimate~\eqref{ergbound2} holds for a subset of IETs whose complement has positive Hausdorff codimension. This observation, together with the discussion following Theorem~\ref{twisbirk}, yields Theorem~\ref{poldecave}. 

\begin{thm} For every irreducible permutation $\pi$ on $d>3$ symbols, which is not of rotation class, there exists a measurable subset $\Delta_{cor}$ of $\Delta=\mathbb{PR}_{+}^{d-1}$, whose complement has positive Hausdorff codimension, and there exists $\alpha'>0$, depending only on the Rauzy class of $\pi$, such that the following holds. For every $\lambda \in \Delta_{cor}$, for every zero average Lipschitz function $f$ and $L^{2}$ function $g$, there exists a constant $C_1(\lambda)>0$ such that,  for all $N\in \mathbb N\setminus\{0\}$, we have
\begin{equation}
    \sum_{n=0}^{N-1} \left|\left\langle f\circ T_{\lambda, \pi}^{n}, g \right\rangle \right|^{2} \leq C_1(\lambda)\|f\|_{L} N^{1-\alpha'} \|g\|_{2}^2 \|f\|_{2}\,.
\end{equation}
\end{thm}
\noindent We now explain in detail the construction outlined in the last paragraph.
\\

\textbf{From translation flows to IETs.} Without loss of generality we may and do assume that $\pi_t^{-1}(d)=\pi_b^{-1}(1)$. Consider the permutation pair $\tilde{\pi}:= (\tilde \pi_t, \tilde \pi_b)$, defined on the set $\mathcal A\setminus \{\pi_t^{-1}(d)\}$ with image $\{1, \dots, d-1\}$ given by 
\begin{equation}
  \tilde{\pi}_t (\beta)=\pi_t(\beta)\,, \quad    \tilde{\pi}_b(\beta)=\pi_b(\beta)-1\,, \quad \text{ for all } \beta \in 
  \mathcal A\setminus \{\pi_t^{-1}(d)\}\,.
\end{equation}
We consider the line defined by $se_{d}+(1-s)\tilde{\lambda} \subset \Delta$, where $\tilde{\lambda} \in \mathbb{PR}_{+}^{d-2}$.
Now let $M$ be a square. We identify the parallel vertical edges with each other. We divide the top edge into segments of lengths given by components of $\tilde{\lambda}$ and in the order dictated by the map $\tilde \pi_t$. We then divide the bottom edge into segments of lengths given by $\tilde{\lambda}_{\tilde{\pi}}:=(\tilde{\lambda}_{\tilde{\pi}^{-1}_b(1)},\dots, \tilde{\lambda}_{\tilde{\pi}^{-1}_b(d-1)})$. We then identify each segment on the top edge with the corresponding one on the bottom. The result is a translation surface $S_{\tilde{\lambda}}$. The return map to the diagonal of the square $M$ of the translation flow in direction $\theta$ (as illustrated in figure \ref{figureIET}) is the interval exchange $T_{\lambda, \pi}$, where 
\begin{equation}
    \lambda= \frac{\tan(\theta)}{1+\tan(\theta)}e_d+\frac{1}{1+\tan(\theta)}\tilde{\lambda}.
\end{equation}
This map is Lipschitz, and thus the Hausdorff dimension of the discarded subset of the line segment is bounded above by that of the discarded subset of the circle of directions. Therefore, the Hausdorff codimension of the set of IETs for which \eqref{ergbound2} does not hold is positive.

\begin{rmk}
An alternative approach to achieve a similar result would be to use Roth type IETs and the results of \cite{10.2307/20161260} on the solutions to the cohomological equation combined with a polynomial bound on the deviations of ergodic averages for piecewise constant roof functions (see \cite{Zorich97deviationfor}). The last ingredient to show that the exceptional set of IETs does not have full Hausdorff dimension would then be to use the results of \cite{article} and \cite{2017arXiv171110542A}.
\end{rmk}

\subsection{Upper bound for rotation class IETs}
\label{upper bound rotation class}
We recall that an IET (or rather its corresponding permutation) is said to be of rotation class if there exists a rotation permutation in its Rauzy class. Throughout this and the next subsection we will assume that $\pi$ is of rotation class but is not a rotation itself. Such IETs are in particular of type W, according to the following
\begin{definition} \label{def:typeW}  An irreducible permutation is 
of type W if the vector $h=(1,\cdots,1) \notin H(\pi) \subset \mathbb R^d$, the subspace defined in section \ref{VZR}. An Interval Exchange Transformation is of type W if its permutation is of type W. 
\end{definition}
The notion of a type W permutation and IET was introduced with a different, but equivalent, combinatorial definition in \cite{CN01}.
However, type W IETs in the sense of the above definition were considered already by Veech~\cite{10.2307/2374396}, who applied his criterion to IETs of type W and showed that they are typically weakly mixing. He proceeded by showing that there exists a vector $b$ with integral coordinates belonging to the subspace $N(\pi)$ (see 
subsection~\ref{VZR} for the definition of $N(\pi)$) so that for $h=(1, \cdots, 1)$, for all $k\in \mathbb{N}$, and all $t \in (0,1)$
\begin{equation}
    \left\langle A_k(x). th, b\right \rangle= t,
\end{equation}
where $A_k$ may be taken to be any inducing of the Rauzy-Veech cocycle. 
This implies that there exists $M_0>0$ such that for the vector $\vec{s}=(1,\cdots, 1)$ we have
\begin{equation}
    \underset{v \in GR(\zeta)}{\max}\|\omega |\zeta^{[k]}(v)|_{\vec{s}}\|^{2}_{\mathbb{R}/\mathbb{Z}}> M_0\|\omega\|_{\mathbb{R}/\mathbb{Z}}^{2}.
\end{equation}
Note that, by Proposition~\ref{genbound},

\begin{equation}
 \begin{split}
     \left|S_{R}^{(x,t)}(f^{(\ell)}, \omega)\right| &\leq C(\vec{s}, Q) \cdot \|f^{(\ell)}\|_{\infty} 
     \\ &\times \left(R^{1/2}+ R\prod_{\ell+1 \leq k \leq \frac{\log(R)}{2\theta_1}} \left(1-c_{1}.\underset{v \in GR(\zeta)}{\max}\|\omega |\zeta^{[k]}(v)|_{\vec{s}}\|^{2}_{\mathbb{R}/\mathbb{Z}} \right)\right)\,.
\end{split}
\end{equation}
For $R \geq e^{4\ell \theta_1}$
\begin{equation}
\begin{split}
    R\prod_{\ell+1 \leq k \leq \frac{\log(R)}{2\theta_1}} \left(1-c_{1}.\underset{v \in GR(\zeta)}{\max}\|\omega |\zeta^{[k]}(v)|_{\vec{s}}\|^{2}_{\mathbb{R}/\mathbb{Z}} \right) &\leq R \exp\left(-c_1M_0(\frac{\log R}{4\theta_1}) \|\omega\|^{2}_{\mathbb{R}/\mathbb{Z}}\right)\\
    & \leq R^{1-M\|\omega\|_{\mathbb{R}/\mathbb{Z}}^{2}},
    \end{split}
\end{equation}
for some small constant $M>0$. Therefore
\begin{equation}
    \left|S_{R}^{(x,t)}(f^{(\ell)}, \omega)\right| \leq 2C(\vec{s}, Q). \|f^{(\ell)}\|_{\infty} R^{1-M\|\omega\|_{\mathbb{R}/\mathbb{Z}}^{2}}.
\end{equation}
We now fix $\gamma= \frac{1}{16}$ and observe that by formula \eqref{funcapprox} for $R\geq e^{\ell_0\theta_1 \gamma^{-1}}$ we get 
\begin{equation}
     \left|S_{R}^{(x,t)}(f, \omega)-S_{R}^{(x,t)}(f^{(\ell)}, \omega)\right| \leq R \cdot \|f\|_{L}e^{-3\ell \theta_1/4} \leq \|f\|_{L} R^{1-\gamma}\,.
\end{equation}
Therefore, there exists $C(\mathrm{a})>0$ depending on $\mathrm{a} \in \Omega'_{\mathbf{q}}$ (indeed on $\ell_0(\mathrm{a})$ as defined in Definition~\ref{defi}) such that 
\begin{equation}
    \left|S_{R}^{(x,t)}(f^{(\ell)}, \omega)\right| \leq C(\mathrm{a})\|f\|_{L}R^{1-M'\|\omega\|_{\mathbb{R}/\mathbb{Z}}^{2}}.
\end{equation}
We have shown the following

\begin{thm} Let $\pi$ be a permutation on $d$ letters that is of rotation class but is not a rotation itself. Then, there exists a set $\Delta_{spec} \subset \Delta$ such that for every $\lambda \in \Delta_{spec}$, $T_{\lambda, \pi}$ satisfies the following. There exists a constant $C(\lambda)>0$ such that, for every weakly Lipschitz function $f:I \to \mathbb{R}$ and every $\omega \in (0,1)$, we have
\begin{equation}
\label{LogTwistedBound}
    \left| \sum_{n=0}^{N-1} e^{-2\pi i n\theta }f(T_{\lambda, \pi}^{n}(x))\right| \leq C(\lambda)  \|f\|_{L} N^ {  1-M\|\theta\|^{2}_{\mathbb{R}/\mathbb{Z}} } \, ,
\end{equation}
where $M>0$ is a universal constant that only depends on the Rauzy class of the permutation.

\end{thm}

By Denjoy-Koksma inequalities and the metric theory of rotations (see Khinchine~\cite{MR0161833}, \cite{PMIHES_1979__49__5_0}), for almost every rotation number and therefore for almost every IET of rotation class, for every bounded variation function $f:I \to \mathbb{R}$ of zero-average, for every $\delta>0$ and for every $x\in I$, we have
\begin{equation}
    \left| \sum_{n=0}^{N-1} f(T^{n}_{\lambda,\pi}(x)) \right| 
    \leq C_{\delta} Var(f)\log N (\log\log N)^{1+\delta},
\end{equation}
where $Var(f)$ denotes the total variation of $f$ and $C_\delta>0$ is a positive constant only depending on $\delta$ and $(\lambda, \pi)$.
As in the subsection~\ref{upper bound for non-rotation type} we may write
\begin{equation}
\begin{split}
    I:=\sum_{n=0}^{N-1} \left|\left\langle U^{n}(f), g \right\rangle\right|^{2}
    &= \sum_{n=0}^{N-1}{\big\langle U^{n}(f), g\big\rangle} \int_{\mathbb{R}/\mathbb{Z}} e^{2\pi\imath n \theta} d\sigma_{f,g}(\theta)
    \\&=\int_{\mathbb{R}/\mathbb{Z}}
 \big\langle \sum_{n=0}^{N-1}e^{2\pi\imath  n \theta} U^{n}(f), g \big\rangle d\sigma_{f,g}(\theta)
 =I_{\epsilon}^{+}+I_{\epsilon}^{-} \,,
\end{split}
\end{equation}
where $I_{\epsilon}^{+}$ and $I_{\epsilon}^{-}$ are as defined in formulas~\eqref{int+} and \eqref{int-}. 

By a variation of Lemma~\ref{Twisted Birkhoff Spectral} 
\begin{equation}
    \left|\sigma_{f,g}\left((-\epsilon, \epsilon)\right)\right| \leq \frac{\pi^2 C_{\delta}}{4} Var(f) \|g\|_2 \epsilon \log (\frac{1}{\epsilon}) (\log \log (\frac{1}{\epsilon}))^{1+\delta}.
\end{equation}

Therefore
\begin{equation}
    I_{\epsilon}^{-} \leq \widehat{C}_{\delta} Var(f) \|f\|_2 \|g\|_{2}^{2} \epsilon \log(\frac{1}{\epsilon}) (\log \log(\frac{1}{\epsilon}))^{1+\delta}N,
\end{equation}
and, by \eqref{LogTwistedBound} and the fact that $\left|\sigma_{f,g}\left((0,1)\right)\right| \leq \|f\|_2 \|g\|_2$,
\begin{equation}
    I_{\epsilon}^{+} \leq C(\lambda) \|f\|_{L}\|f\|_{2}\|g\|_{2}^{2}N^{1-M\epsilon^2}.
\end{equation} 

\noindent 
It may be easily verified that the function
\begin{equation}
\label{function_of_u}
    u \mapsto \frac{u^2\left(\log u- \log \log u- (1+\delta) \log \log \log(u)\right)}{M}
\end{equation}
is a diverging continuous function of $u>e$ and therefore it attains all positive values larger than a certain constant. For $N$ sufficiently large, we take $u$  such that the value of the function in~\eqref{function_of_u} is equal to $\log N$ and we let $\epsilon:= \frac{1}{u}$. Then the powers of $N$ in $I_{\epsilon}^{+}$ and $I_{\epsilon}^{-}$ match. Therefore
\begin{equation}
I \leq (C(\lambda)+C_{\delta}) \|f\|_{L}\|f\|_{2}\|g\|_2^{2} N^{1-M\epsilon^2}
\end{equation}
for this specific choice of $\epsilon>0$. By the above choice of $u$ we have
\begin{equation}
    \frac{u^2 \log u}{2M} <\log N< u^{3}  \Rightarrow N^{-M/u^{2}}< \frac{1}{\sqrt{u}}< \frac{1}{(\log N)^{1/6}},
\end{equation}
and therefore we have shown the following
\begin{thm}
Let $d>2$ be an integer and $\pi$ be a rotation class permutation that is not a rotation. Then for almost every $\lambda \in \mathbb{P}^{d-1}$, for every zero-average Lipschitz function $f$, and every $L^{2}$ observable $g$ we have

\begin{equation}
    \sum_{n=0}^{N-1} \left|\big \langle f \circ T_{\lambda, \pi}^{n}, g \big \rangle\right|^{2}  \leq C_2(\lambda) \|f\|_{L} \|f\|_{2} \|g\|_{2}^{2}\frac{N}{(\log N)^{1/6}} \,,
\end{equation}
where $C_2(\lambda)$ is a positive constant. 
\end{thm}

\subsection{Lower bound for rotation class IETs}
\label{lower bound rotation class}

We show that a polynomial decay for the Ces\`aro averages of correlations is not possible for a typical interval exchange of rotation class. Note that a typical interval exchange of rotation class may be viewed as the time-one map of a special flow over an irrational rotation with a piecewise constant roof function with integer values. Let us denote the interval exchange transformation by $T$ and denote the base dynamics by $R_\theta$. We also assume that the base interval $I_0$ is  not of length~$1$ and we let $r:I_0 \to \mathbb{N}$ be the roof function. Let  $I_\alpha$'s for $\alpha$'s belonging to a set of indices $\mathcal{A}$ denote the intervals of continuity of $r$. 
Let $a$ denote the length of $I_0$. 
Then $R_{\theta}:I_0 \to I_0$ is defined by
\begin{equation}
    R_{\theta}(x)=x+a\theta \; \quad \text{mod} \;a \,,
\end{equation}
which is conjugate to the map $x \mapsto x+\theta$ mod $1$. 

By Denjoy-Koksma inequality, and the metric theory of rotations, for almost every $\theta \in \mathbb R$ we have that the deviation of Birkhoff sums of sufficiently regular functions (bounded variation is enough) from their average grows no faster than a logarithmic function of time, i.e., for every $x \in I_0$,
\begin{equation}
\label{DKDEV}
    \left|\sum_{i=0}^{k-1} r(R^{i}_\theta(x))- k\int_{I_{0}} r \right| \leq C_{\epsilon} Var(r) \log k (\log \log k)^{1+\epsilon} \,,
\end{equation}
for any $\epsilon>0$ and some constant $C_\epsilon>0$ depending only on $\epsilon$ (see for instance, \cite{MR0161833}). Note that in the above inequality $Var(r)$ denote the total variation of $r$. 

Note that the action of $T: [0,1) \to [0,1)$ may be viewed as follows.
There exists a piecewise smooth embedding 
$$i:[0,1)\to \{(x,t) \vert x\in I_0, 0\leq t <r(x), t \in \mathbb N\}$$ such that, for $(x,t)\in i[0,1)$, except for a finite set of points, we have 
\begin{equation}
   (i \circ T \circ i^{-1})(x,t)=
\begin{cases}
(x, t+1) & \text{if} \; t+1< r(x),\\
(R_{\theta}(x), 0) & \text{if} \; t+1=r(x).
\end{cases}
\end{equation}
In other terms, for $y \in [0,1)$ there exists a unique $x \in I_0$ and a unique integer $0 \leq t<r(x)$ such that $y$ corresponds to $(x, t)$ in the special flow setting.

Now let $J_{N} \subset I_0$ be the interval $[s, 2s)$ of size $s=\frac{c_0 a}{ \log N(\log \log N)^{1+\epsilon}}$, 
where $c_0$ is a constant that will be determined later. We now look at the iterates of this interval under the interval exchange transformation $T$. Assume that $(x,0)\in J_N$ and $T^{m}(x,0)$ also belongs to the base interval $I_0$. Then, we must have
\begin{equation}
    \sum_{i=0}^{\ell-1}r(R_\theta^{i}(x))=m,
\end{equation}
for some positive integer $\ell \leq m$ (as the left-hand side corresponds to the $\ell$-th return time of $x$ to the base interval $I_0$). Then by \eqref{DKDEV} we have  
\begin{equation}
    \left|m-\ell/a \right| \leq C\log\ell (\log \log\ell)^{1+\epsilon}\leq C \log m(\log \log m)^{1+\epsilon},
\end{equation}
as it may be seen either by Kac's lemma or simple direct computation that $\int_{I_{0}}r= \frac{1}{a}$. Therefore, there are at most $A\log m (\log \log m)^{1+\epsilon}$ (for $A=2aC+1$) different possibilities for $\ell$ as $\ell$ is an integer. Thus, 
\begin{equation}
    T^{m}(J_N) \cap I_0 \subset \bigcup_{\ell= 
    \lfloor am\rfloor -\lfloor aC\log m (\log \log m)^{1+\epsilon} \rfloor}^{ \lfloor am \rfloor+ \lceil (1+aC)\log m (\log \log m)^{1+\epsilon} \rceil } R_{\theta}^{\ell}(J_N):=H_m(J_N).
\end{equation}
Now we take 
\begin{equation}
S:=\big\{ n\in \mathbb{N}| \; 1 \leq n \leq N, (\lfloor aN \rfloor -\lfloor an \rfloor)a\theta \in [0, |J_N|/2)\; \text{mod} \; a \big\}.
\end{equation}
Note that by definition, we then have
\begin{equation}
    \forall n \in S, \; H_n(J_N) \subset H_N(J_N'),
\end{equation}
where $J_N'=[0,3s)$. Therefore, we have
\begin{equation}
\label{uni}
    \bigcup_{n\in S} H_n(J_N) \subset H_N(J_N').
\end{equation}
Now we need an estimate on the cardinality of $S$. To this end, let $S_{k}$ be the following set
\begin{equation}
    S_k:=\big \{ n \in \mathbb{N}| 1 \leq n \leq k; na\theta \in [0,J_N/2)\; \text{mod} \; a\big\}.
\end{equation}
By applying \eqref{DKDEV} to $1_{[0,J_N/2)}$ in place of $r$ we get that for large enough $N$ and for $k=\lfloor aN \rfloor $,
\begin{equation}
    \# S_k \geq \frac{k|J_N|}{2} -C \log k (\log \log k)^{1+\epsilon} \geq \frac{k|J_N|}{3} \geq \frac{aN|J_N|}{4}.
\end{equation}
Looking at the sequence $\lfloor an \rfloor $ for $1\leq n \leq N$, we observe that every integer value between $1$ and $\lfloor aN \rfloor$ occurs in this sequence at least once. Therefore, we have the following lower bound 
\begin{equation}
\label{lowcard}
    \# S \geq \#S_{\lfloor aN \rfloor} \geq \frac{aN|J_N|}{4}.
\end{equation}
It may be readily verified that the total length of the intervals contained in $H_N(J_N')$ is bounded by $3A\log N (\log \log N)^{1+\epsilon}|J_N|$. Thus, if we choose $c_0$ to be $\frac{1}{30A}$ (note that $A$ is independent of $N$ and only depends on the roof function $r$ and $\theta$) we get that
\begin{equation}
    |H_N(J_N')| \leq \frac{a}{10}. 
\end{equation}
Thus, the complement of $H_N(J_N')$ in $I_0$ will be a union of at most $A\log N (\log \log N)^{1+\epsilon}+2$ intervals whose total length is at least $\frac{9a}{10}$. Therefore, we may find an interval $J_N''$ of the same length $s$ as $J_N$ in the complement of $H_N(J_N')$. By construction, the intervals $J_N$ and $J_N''$ have the property that
\begin{equation}
    J_N'' \cap \bigcup_{m\in S} T^{m}(J_N) =\emptyset.
\end{equation}
We pick two Lipschitz functions $f$, supported in $J_N$, and $g$, supported in $J_N''$ and bound the Ces\`aro average of their correlations until time $N$ from below. By \eqref{lowcard},

\begin{equation}
\label{lowces}
\begin{aligned}
    Q_N(f, g):=\sum_{n=0}^{N-1} &\left|\int_I f \circ T^{n}(y) g(y) dy - \int_I f \int_I g \right|^{2} \\
    &\geq  \|f\|_{L^{1}}^{2} \|g\|_{L^{1}}^{2} \# S \geq c_1 \frac{N\|f\|_{L^1}^{2} \|g\|_{L^1}^{2}}{\log N (\log \log N)^{1+\epsilon}}\,,
\end{aligned}    
\end{equation}
for some constant $c_1>0$ not depending on $N$. We may now further assume that $f$ and $g$ are so that their range is $[0,1]$, and satisfy the following properties:
\begin{enumerate}
    \item They are equal to $1$ in the interval of length $s/2$ centered at the center of the intervals $J_N$ and $J_N''$, 
    \item They are equal to zero in the union of two intervals of length $s/10$ at the ends of each of the intervals,
    \item Their Lipschitz norms are bounded above by $15/s$.
\end{enumerate}
The above assumptions on $f$ and $g$ will then yield that 
\begin{equation}
    \|f\|_{L^{1}}^{2}\|g\|_{L^{1}}^{2} \geq (\frac{s}{2})^{4}\geq \|f\|_{L}^{2}\|g\|_{L}^{2} \frac{s^{8}}{10^{6}} 
\end{equation}
Then, the lower bound~\eqref{lowces} implies the following
\begin{thm}
There exists a constant $c_2>0$ such that, for all sufficiently large $N\in \mathbb N$, there exist Lipschitz observables $f$ and $g$
(depending on $N$) such that 
\begin{equation}
    Q_N(f, g) \geq c_2 \frac{N\|f\|_{L}^{2}\|g\|_{L}^{2}}{(\log N (\log \log N)^{1+\epsilon})^{9}}.
\end{equation}
\end{thm}
Therefore the norm of the quadratic form $Q_N/N$ has a logarithmic lower bound and in particular it cannot satisfy any polynomial upper bound. 
\begin{rmk}
Although, the above example is provided only for a full measure subset of irrational rotation numbers, using exactly the same construction one can show that as the deviation of ergodic averages are subpolynomial no polynomial upper bound can hold for any IET of rotation class.  
\end{rmk}

\nocite{*}
\printbibliography
\end{document}